\documentclass[journal,twoside,web]{ieeecolor}
\usepackage[usenames,dvipsnames]{xcolor}
\usepackage{generic}
\usepackage{cite}
\usepackage{graphicx}
\usepackage{textcomp}

\usepackage{times,pslatex}
\usepackage{amsfonts}
\usepackage{amssymb}
\usepackage{amsmath}
\usepackage{graphicx}
\usepackage[all]{xy}
\usepackage{bm}
\usepackage[section]{algorithm}
\let\labelindent\relax
\usepackage[shortlabels]{enumitem}

\usepackage[colorlinks, linkcolor=reference,
 citecolor=citation,
         urlcolor=e-mail]
      {hyperref}
\usepackage{csquotes}

\definecolor{conditional}{rgb}{0,1,0}
\definecolor{e-mail}{rgb}{0,.40,.80}
\definecolor{reference}{rgb}{.20,.60,.22}
\definecolor{mrnumber}{rgb}{.80,.40,0}
\definecolor{citation}{rgb}{.20,.60,.22}

\usepackage{bbm}

\usepackage{tikz}
\usetikzlibrary{arrows,decorations.markings}

\newtheorem{thm}{Main Result}
\newtheorem{theorem}{Theorem}
\newtheorem{problem}{Problem}
\newtheorem{lemma}{Lemma}
\newtheorem{proposition}{Proposition}
\newtheorem{corollary}{Corollary}
\newtheorem{notation}{Notation}
\newtheorem{definition}{Definition}
\newtheorem{remark}{Remark}
\newtheorem{example}{Example}

\DeclareMathOperator{\ord}{ord}

\newcommand{\Q}{\mathbb{Q}}

\newcommand{\CC}{\mathbb{C}}

\newcommand{\Frac}{\text{Frac}}
\DeclareMathOperator{\In}{In}
\DeclareMathOperator{\Out}{Out}
\DeclareMathOperator{\Leak}{Leak}

\def\BibTeX{{\rm B\kern-.05em{\sc i\kern-.025em b}\kern-.08em
    T\kern-.1667em\lower.7ex\hbox{E}\kern-.125emX}}
\begin{document}
\title{Input-output equations and identifiability of linear ODE models}
\author{Alexey Ovchinnikov, Gleb Pogudin, and Peter Thompson
\thanks{
This work was partially supported by the NSF grants CCF-1563942, CCF-1564132, CCF-1319632, DMS-1760448, CCF-1708884, DMS-1853650, DMS-1853482; NSA grant \#H98230-18-1-0016; and CUNY grants PSC-CUNY \#69827-0047, \#60098-00 48.}
\thanks{A. Ovchinnikov is with CUNY Queens College, Department of Mathematics,
65-30 Kissena Blvd, Queens, NY 11367, USA and 
CUNY Graduate Center, Ph.D. Programs in Mathematics and Computer Science, 365 Fifth Avenue,
New York, NY 10016, USA (e-mail: aovchinnikov@qc.cuny.edu). }
\thanks{G. Pogudin was with Courant Institute of Mathematical Sciences, New York University, New York, NY 10012, USA and the Higher School of Economics, Faculty of Computer Science, Moscow, 109028, Russia. He is 
now with LIX, CNRS, \'Ecole Polytechnique, Institute Polytechnique de Paris, 1 rue Honor\'e d'Estienne d'Orves, 91120, Palaiseau, France (e-mail: gleb.pogudin@polytechnique.edu).}
\thanks{P. Thompson was with CUNY Graduate Center, Ph.D. Program in Mathematics, 365 Fifth Avenue,
New York, NY 10016, USA
 (e-mail: peterthompsonmath@gmail.com, peter.thompson@liu.se).}}

\maketitle

\begin{abstract}
Structural identifiability is a property of a differential model with parameters that allows for the parameters to be determined from the model equations in the absence of noise. The method of input-output equations is one  method for verifying structural identifiability. 
This method stands out in its importance because the additional insights it provides can be used to analyze and improve models.
However, its complete theoretical grounds and applicability are still to be established. A subtlety and key for this method to work  correctly is  knowing
 whether
the coefficients of these equations are identifiable. 

In this paper, to address this, we prove  identifiability of the coefficients of input-output equations for types of differential models that often appear in practice, such as linear models with one output and linear compartment models in which, from each compartment, one can reach either a leak or an input. 
This shows that  checking identifiability via input-output equations for these models is legitimate and, as we prove, that the field of identifiable functions is generated by the coefficients of the input-output equations.
Finally, we exploit  
 a
connection between input-output equations and the transfer function matrix to show that, for a linear compartment model with an input and strongly connected graph, the field of all identifiable functions is generated by the coefficients of the transfer function matrix even if the 
initial conditions are generic.
\end{abstract}


\section{Introduction}

\subsection{Background}\label{sec:background}
Structural global identifiability (in what follows, we will say just ``identifiability'' for simplicity) is a property of a differential model with parameters that allows for the parameters to be uniquely determined from the model equations, noiseless data and sufficiently exciting inputs (also known as the persistence of excitation, see~\cite{LJ94, V18, XiaMoog}). Performing  identifiablity analysis is an important first step in evaluating and, if needed, adjusting the model. 
There are different approaches to assessing  identifiability (see \cite{comparison,Hong,Villaverde2019} for descriptions of methods).  If structural identifiability is established, one can assess practical identifiability before doing reliable  parameter identification~\cite{DiStefano2014,WP1997}.

There is a relaxed version of identifiability, namely, local identifiability.
It refers to the possibility of determining finitely many feasible parameter values.
There are efficient algorithms~\cite{KAJ2012,S2002} for checking whether a given function of parameters is locally identifiable.
To the best of our knowledge, there are no complete and efficient algorithms for finding all locally identifiable functions of parameters (see~\cite[page~7]{Villaverde2016} for a partial algorithm), a key to efficient model reparametrization for improving the model.

\subsubsection{How the errors that we prevent occur in existing methods}
One of the
approaches, which is widely used, is based on \textbf{input-output equations}~\cite{BD16, Saccomani2003,Meshkat14res,Meshkat18,BEC2013,DAISY,DAISYIFAC,DAISYMED,COMBOS,COMBOS2,MRS2016, graph_ident}, and has appeared in
software packages such as
 COMBOS, DAISY, and their successors. 
 An existing 
 challenge is to understand the a priori applicability of the method, as the above software packages make incorrect identifiability conclusions for some models. We address this challenge in the present paper.
 
 We will now discuss this in more detail.
Roughly speaking,  
 input-output equations
are ``minimal'' equations that depend only on the input and output variables and parameters (see~\cite{KYN2020} for applications other than identifiability).
We will describe a typical algorithm based on this approach using the following linear compartment model as a running example:
\begin{equation}\label{eq:2comp}
    \begin{cases}
      x_1' = -(a_{01} + a_{21})x_1 + a_{12} x_2 + u,\\
      x_2' = a_{21} x_1 - a_{12}x_2,\\
      y = x_2.
    \end{cases}
\end{equation}
In the above system,
\begin{itemize}[leftmargin=7.5mm]
    \item $x_1$ and $x_2$ are 
    the state variables;
    \item $y$ is the output observed in the experiment;
    \item $u$ is the input (control) function to be chosen by the experimenter;
    \item $a_{01}, a_{12}, a_{21}$ are unknown scalar parameters.
\end{itemize}
The question is whether the values of the parameters $a_{01}, a_{12}, a_{21}$ can be determined from $y$ and $u$.
A typical algorithm operates as follows:
\begin{enumerate}[label=\textbf{(\arabic*)}, leftmargin=6mm]
  \item  Find input-output equations, writing
  them as (differential) polynomials in the input and output variables. 
  For~\eqref{eq:2comp}, a calculation shows that the input-output equation is
  \begin{equation}\label{eq:io_example}
    y'' + (a_{01} + a_{12} + a_{21}) y' + a_{01}a_{12}y  - a_{21}u = 0.\hypertarget{assumption}{} 
  \end{equation}
  \item\label{step:assumption}
  Use the following \textbf{Assumption (A)}: 
  
  \begin{displayquote}\emph{a function of parameters is identifiable if and only if it can be expressed as a rational function of the coefficients of the input-output equations}. 
  \end{displayquote}
  In our example, this  amounts to assuming that a function of parameters is identifiable if and only if it can be expressed as a rational function of $a_{01}+a_{12}+a_{21}$, $a_{01}a_{12}$, and $a_{21}$.
  
  One possible rationale behind this assumption is the ``solvability'' condition from~\cite[Remark~3]{Saccomani2003}: due to the ``minimality'' of the input-output equations, one would expect that there exist $N$ and $t_1,\ldots,t_N \in \mathbb{R}$ such that the linear system
  \begin{equation}\label{eq:solvability_example}
  \begin{cases}
    y''(t_1) + c_1 y'(t_1) +c_2y(t_1) + c_3u(t_1) = 0\\
   \quad\vdots\\
   y''(t_N) + c_1 y'(t_N) +c_2y(t_N) + c_3u(t_N) =0
  \end{cases}
  \end{equation}
  in $c_1,c_2,c_3$ has a unique solution in terms of $y(t_i), y'(t_i), y''(t_i), u(t_i)$, $1\leqslant i\leqslant N$, so the coefficients of~\eqref{eq:io_example} are identifiable.
  However, \emph{the assumption is not always satisfied} and, consequently, such $N$ and $t_1,\ldots,t_N$ might not exist at all. 
  This is a reason, e.g., why DAISY may miss the non-identifiability of some of the parameters in those systems.
  An example is given in Section~\ref{sec:ex_nonsolvable} (see also \cite[Example~2.14]{Hong} and~\cite[Sections~5.2 and 5.3]{allident}).
  
  \item  Set up a system of polynomial equations in the parameters setting the coefficients of~\eqref{eq:io_example} equal to new variables,
  \begin{equation}\label{sys:intro_example}
    \begin{cases}
    a_{01} + a_{12} + a_{21} = c_1\\
    a_{01}a_{12}=c_2\\
    -a_{21}=c_3,
    \end{cases}
  \end{equation}
  and verify if~\eqref{sys:intro_example} as a system in the $a$'s with coefficients in the field $\mathbb{C}(c_1,c_2,c_3)$ has a unique solution.
  This can be done, 
   e.g., using Gr\"obner bases.
  Alternatively, for~\eqref{sys:intro_example}, one can see that $a_{21} = -c_3$ can be uniquely recovered, but the values of $a_{01}$ and $a_{12}$ are known only up to exchange due to the symmetry of~\eqref{sys:intro_example} with respect to $a_{01}$ and $a_{12}$.
\end{enumerate}

\subsubsection{Importance of the IO-equation method: finds all identifiable combinations and helps with reparametrization}
Even though there are complete algorithms (that is, not relying on any assumption like \hyperlink{assumption}{Assumption (A)} above) for assessing structural identifiability (see, e.g., \cite{SIAN}), establishing when
the input-output equation method is valid is important
 because:
\begin{itemize}[leftmargin=5mm]
\item This method can produce {\em all identifiable functions} (also referred to as ``true parameters'' in~\cite[Remark~2]{KYN2020}),
not just assess identifiability of specific parameters.
More precisely, 
\cite[Corollary~5.8]{ident-compare}
shows that the field generated by the coefficients of the input-output equations contains all of the identifiable functions.

In  example~\eqref{eq:2comp}, the field of identifiable functions is generated by the coefficients of~\eqref{eq:io_example}, so it is equal to \[\CC(a_{01} + a_{12} + a_{21},\; a_{01} a_{12},\; a_{21}) = \CC(a_{01} + a_{12},\; a_{01} a_{12},\; a_{21}).\]
Generators of the field of identifiable functions can be  used to reparametrize the model~\cite{BD16, BMM2004, MS2014}.

\item This method can be used for  proving general theorems about classes of models~\cite{Meshkat14res,Meshkat18}.
\item For a large class of linear compartment models, there are efficient methods for computing their input-output equations \cite{Meshkat14res,MRS2016,Meshkat18}.
\end{itemize}

\subsection{The problem}\label{subsec:main_problem}
As was described above, the approach to assessing identifiability via input-output equations has been used much in the last three decades and has its own distinctive features.
However, it heavily relies on \hyperlink{assumption}{Assumption~(A)}, which is not always true (see~\cite[Example~2.14]{Hong} and~\cite[Section~5.2]{allident}).
It can be verified by an algorithm~\cite[Section~4.1]{DVJBNP01} and~\cite[Section~3.4]{MXPW11} but is not verified in any implementation we have seen (including~\cite{DAISY, COMBOS}).
\textbf{The general problem} studied in this paper is: 
\begin{displayquote}\emph{to  
determine classes of ODE models that satisfy \hyperlink{assumption}{Assumption~(A)} a priori; consequently, the approach via input-output equations gives correct result for these models}.
\end{displayquote}
Discrepancy between different notions of identifiability is not unusual given the wide range of experimental setups and mathematical tools involved.
We refer the reader to a recent review~\cite{AnstettCollin2020} (see also~\cite{XiaMoog}) presenting a number of notions of identifiability together with some known (in)equivalences between them.
Our work clarifies
this big picture by giving explicit and easy to check (unlike~\cite{ident-compare}) conditions for equivalence of different ways to assess identifiability.

\subsection{Our results}
\emph{The first part} of our results shows that  \hyperlink{assumption}{Assumption (A)} is a priori satisfied for the following classes of models often appearing in practice \cite{AASC1998,BA1970,DiStefano2014,GC1990,GD1985,COMBOS,V1998,WL1981}: 
\begin{itemize}[leftmargin=7.5mm]
    \item linear models with one output (Main Result~\ref{thm:main1});
    \item linear compartment models such that, from every vertex of the graph of the model, at least one leak or input is reachable (Main Result~\ref{thm:main2}).
\end{itemize}
Checking whether the model is of one of these types can be done just by visual inspection.
For instance, as we will see in Example~\ref{ex:lincomp}, each of these theorems is applicable to  model~\eqref{eq:2comp}.
 Main Result~\ref{thm:main1} cannot be strengthened to more than one output if all linear models are allowed, see~Section~\ref{sec:ex_nonsolvable} and, for non-linear systems, see \cite[Lemma~5.1]{second_paper}.

\emph{The second part} is devoted to relaxing the ``minimality'' condition on the input-output equations.
For linear compartment models, elegant relations involving only parameters, inputs, and outputs were proposed in~\cite[Theorem~2]{Meshkat14res}  based on Cramer's rule (see also~\cite[Proposition~2.3]{Meshkat18}).
In general, using these equations instead of the ``minimal'' relations in the algorithm above would give incorrect results  \cite[Remark~3.11]{Meshkat18}.

However, in Main Result~\ref{thm:main3}, we show that, for linear compartment models with an input and  whose  graph is strongly connected, one can use these equations as the input-output equations and obtain the full field of identifiable functions.

Furthermore, we apply Main Result~\ref{thm:main3} to the transfer function method~\cite[page~444]{DiStefanoBook}.
It is known that, in  case of multiple outputs, using only the coefficients of the transfer function matrix (as opposed to the full output transforms) may lead to incorrect  
identifiability conclusions~\cite[Example~10.6]{DiStefanoBook}.
As a corollary of our results (Corollary~\ref{cor:trans}), we show that this is not the case for such linear compartment models.

We state the consequences of our results for algorithms for computing identifiable functions in Section~\ref{sec:applications_to_algorithms} and illustrate the conditions in our main results in Section~\ref{sec:illust_ex}.

%%%%%%%%%%%%%%%%%%

\subsection{Structure of the paper}
Basic notions and notation from differential algebra, identifiability, and linear compartment models are given in Section~\ref{sec:prelim}.
The main results in a brief form are stated in Section~\ref{sec:main} and then  stated and proved in Section~\ref{sec:proofs}.  In Section~\ref{sec:illust_ex}, we illustrate our main results with examples, e.g., showing how existing identifiability approaches could fail.
 The \hyperlink{AppendixA}{appendix} has results we use relating the notions used in the paper for linear models to the corresponding notions for nonlinear systems.

%%%%%%%%%%%%%%%%%%%%%%%%%%%%%%%%%%%%%%%%%%%%%%%%%%%%%

\section{Preliminaries}\label{sec:prelim} In this section, we recall the notation/notions found in the literature and introduce our own  notation/notions 
to state our main results in Section~\ref{sec:main}.
All  fields 
have characteristic zero.

%%%%%%%%%%%%%%%%%%%%%%%%%%%%%%%%%%%%%%%

\subsection{Identifiability of linear models}\label{sec:prelim_ident}

\hypertarget{vars}{}
Fix positive integers $\lambda$, $n$, $m$, and $\kappa$ for the remainder of the paper.
Let $\bm{\mu} = (\mu_1,\ldots,\mu_\lambda)$, $\mathbf{x} = (x_1, \ldots, x_n)$, $\mathbf{y} = (y_1, \ldots, y_m)$, and $\mathbf{u} = (u_1,\ldots,u_\kappa)$.
Consider a system of ODEs \hypertarget{Sigma}{}
\begin{equation}\label{eq:sigma_lin}
\Sigma = 
\begin{cases}
\mathbf{x}' = \mathbf{f}(\mathbf{x}, \bm{\mu}, \mathbf{u}), \\
\mathbf{y} = \mathbf{g}(\mathbf{x}, \bm{\mu}, \mathbf{u}),\\
\mathbf{x}(0) = \mathbf{x}^\ast,
\end{cases}
\end{equation}
where $\mathbf{f} = (f_1, \ldots, f_n)$ and $\mathbf{g} = (g_1, \ldots, g_m)$ are tuples of polynomials in $\mathbf{x},\mathbf{u}$ over $\CC(\bm{\mu})$ of \emph{degree at most one}.

For a rational function $h(\bm{\mu}) \in \CC(\bm{\mu})$, we will define two notions of identifiability: {\em identifiability} and {\em IO-identifiability}.  
The former is meaningful from the modeling standpoint; 
the latter is what the algorithm outlined in the introduction 
checks.

\subsubsection{Identifiability}
We 
 fix
notation 
 to give rigorous definitions:

\begin{notation}[Auxiliary analytic notation]
\begin{enumerate}[label = (\alph*),leftmargin=7.5mm]
\item[]
    \item Let $\CC^\infty(0)$ denote the set of all functions that are complex analytic in some neighborhood of $t = 0$.
    \item Let $\Omega \subset \mathbb{C}^{\lambda}$ be the complement to the set where at least one of the denominators of the coefficients of~\eqref{eq:sigma_lin} in $\mathbb{C}(\bm{\mu})$ vanishes.
    \item For every $h \in \CC(\bm{\mu})$, we set
    \[
    \Omega_h := \CC^n \times \{\hat{\bm{\mu}} \in \Omega \mid h(\hat{\bm{\mu}}) \text{ well-defined}\} \times (\CC^\infty(0))^\kappa.
    \]
    \item For $(\hat{\mathbf{x}}^\ast, \hat{\bm{\mu}}, \hat{\mathbf{u}})$  such that $\hat{\bm{\mu}} \in \Omega$, let $X(\hat{\mathbf{x}}^\ast, \hat{\bm{\mu}}, \hat{\mathbf{u}})$ and $Y(\hat{\mathbf{x}}^\ast, \hat{\bm{\mu}}, \hat{\mathbf{u}})$ denote the unique solution over $\CC^\infty(0)$ of the instance of $\Sigma$ with $\mathbf{x}^\ast = \hat{\mathbf{x}}^\ast$, $\bm{\mu} = \hat{\bm{\mu}}$, and $\mathbf{u} = \hat{\mathbf{u}}$ (see~\cite[Theorem~2.2.2]{Hille}).
    
    \item For any positive integer $s$, a subset $U \subset \CC^s$ is called \emph{Zariski open} if there exists a polynomial $P$ on $\CC^s$ such that $U$ is the complement to the zero set of $P$.
    \item For any positive integer $s$, a subset $U \subset (\CC^\infty(0))^s$ is called \emph{Zariski open} if there exists a polynomial $P$ in $z_1, \ldots, z_s$ and their derivatives such that \[U = \{ \hat{\bm{z}} \in (\CC^\infty(0))^s \mid P(\hat{\bm{z}})|_{t = 0} \neq 0\}.\]
    \item For any positive integer $s$ and $X = \CC^s$ or $(\CC^\infty(0))^s$, the set of all nonempty Zariski open subsets of $X$ will be denoted by $\tau(X)$.
\end{enumerate}
\end{notation}

\begin{definition}[Identifiability, see {\cite[Definition~2.5]{Hong}}]
We say that $h(\bm{\mu}) \in \CC(\bm{\mu})$ is \emph{identifiable} if
\begin{align*}
  \exists \Theta \in \tau(\CC^n &\times \CC^\lambda) \; \exists U \in \tau((\CC^\infty(0))^\kappa)\;\\
  \forall (\hat{\mathbf{x}}^\ast, \hat{\bm{\mu}}, \hat{\mathbf{u}}) \in (\Theta &\times U) \cap \Omega_h \quad |S_h(\hat{\mathbf{x}}^\ast, \hat{\bm{\mu}}, \hat{\mathbf{u}})| = 1,\\ 
 \text{where}\ \  S_h(\hat{\mathbf{x}}^\ast, \hat{\bm{\mu}}, \hat{\mathbf{u}}) :&= \{h(\tilde{\bm{\mu}}) \mid  \exists\;(\tilde{\mathbf{x}}^\ast, \tilde{\bm{\mu}}, \hat{\mathbf{u}}) \in \Omega_h  \\
  &\text{ such that }\; Y(\hat{\mathbf{x}}^\ast, \hat{\bm{\mu}}, \hat{\mathbf{u}}) = Y(\tilde{\mathbf{x}}^\ast, \tilde{\bm{\mu}}, \hat{\mathbf{u}}) \}.
\end{align*}
The field $\{h \in \mathbb{C}(\bm{\mu})\mid h \text{ is identifiable}\}$ will be called \emph{the field of identifiable functions}.
\end{definition}

\subsubsection{Input-output identifiability}
The notion of IO-identifiability can be defined for systems with rational right-hand side (see Section~\ref{sec:ident_general} from the Appendix).
Here we give a specialization of the general definition to the linear case (the equivalence of 
Definition~\ref{def:ioid_lin} and
Definition~\ref{def:ioid} 
restricted to the linear case is established in Proposition~\ref{prop:linear_equiv}). 
For this, we will first recall several standard notions from differential algebra:

\begin{notation}[Differential rings and ideals]\label{not:lin_diffalg}
\begin{enumerate}[label = (\alph*),leftmargin=6mm]
  \item[]
  \item A {\em differential ring} $(R,\delta)$ is a commutative ring with a derivation $':R\to R$, that is, a map such that, for all $a,b\in R$, $(a+b)' = a' + b'$ and $(ab)' = a' b + a b'$. 
  \item The {\em ring of differential polynomials} in the variables $x_1,\ldots,x_n$ over a field $K$ is the ring $K[x_j^{(i)}\mid i\geqslant 0,\, 1\leqslant j\leqslant n]$ with a derivation defined on the ring by $(x_j^{(i)})' := x_j^{(i + 1)}$. 
  This differential ring is denoted by $K\{x_1,\ldots,x_n\}$.
  \item For a differential polynomial $P \in K\{x_1, \ldots, x_n\}$ and $1 \leqslant i \leqslant n$, \emph{the order of $P$ with respect to $x_i$} is the order of the highest derivative of $x_i$ appearing in $P$ ($-\infty$ if $x_i$ does not appear in $P$).
  It is denoted by $\ord_{x_i} P$.
  \hypertarget{DiffIdeal}{}
  \item An ideal $I$ of a differential ring $(R,\delta)$ is called a {\em differential ideal} if, for all $a \in I$, $\delta(a)\in I$. For $F\subset R$, the smallest differential ideal containing set $F$ is denoted by $[F]$.
  \hypertarget{ISigma}{}
    \item  
   For $\Sigma$ as in~\eqref{eq:sigma_lin},  
     let
    $I_\Sigma=[\mathbf{x}'-\mathbf{f},\mathbf{y}-\mathbf{g}] \subset \CC(\bm{\mu})\{\mathbf{x},\mathbf{y},\mathbf{u}\}$ 
   be
    the differential ideal of $\Sigma$.
    Informally,  
    $I_\Sigma$
    is the ideal of all relations 
     among 
    components of a generic solution of~$\Sigma$.
\end{enumerate}
\end{notation}

\begin{definition}[a full set of input-output equations]\label{def:ioeq}
For $\Sigma$ as in~\eqref{eq:sigma_lin}, a tuple $(p_1, \ldots, p_m)$ of differential polynomials from $\CC(\bm{\mu})\{\mathbf{y}, \mathbf{u}\}$ is called \emph{a full set of input-output equations} if there 
 is
an ordering of the output variables, which we will assume  to be $y_{1} < y_{2} < \ldots < y_{m}$ to simplify notation, such that 
\begin{enumerate}[label = (\arabic*),leftmargin=6mm]
    \item $p_1$ is the linear differential polynomial in $y_{1}$ and $\mathbf{u}$ in $I_\Sigma$ of the smallest possible order in $y_{1}$ such that the coefficient of the highest derivative of $y_{1}$ is one.
    \item For every $\ell > 1$, $p_\ell$ is the linear differential polynomial in $y_{1}, \ldots, y_{\ell}$ and $\mathbf{u}$ in $I_\Sigma$ such that 
    \begin{itemize}[leftmargin=4mm]
        \item $\ord_{y_{j}} p_\ell < \ord_{y_{j}} p_j$ for every $1 \leqslant j < \ell$;
        \item the coefficient of the highest derivative of $y_{\ell}$ in $p_\ell$ is 1;
        \item $\ord_{y_{\ell}} p_\ell$ is the smallest possible.
    \end{itemize}
\end{enumerate}
\end{definition}
\hypertarget{idfield}{}
\begin{definition}[IO-identifiable function]\label{def:ioid_lin}
    For a system $\Sigma$, consider a full set $E$ of input-output equations.
    Then the subfield $k$ of $\CC(\bm{\mu})$ generated by the coefficients of $E$ over $\mathbb{C}$
    is called \emph{the field of input-output identifiable (IO-identifiable) functions}.  
    We call $h \in \CC(\bm{\mu})$  \emph{IO-identifiable} if $h \in k$.
\end{definition}

\begin{remark}
    Proposition~\ref{prop:linear_equiv} establishes the equivalence of this definition to Definition~\ref{def:ioid}, which is applicable to a general rational ODE systems. 
    Proposition~\ref{prop:linear_equiv} also implies that the field of input-output identifiable functions does not depend on the choice of a full set of input-output equations.
\end{remark}
    
For examples of input-output equations and IO-identifiable functions, see Section~\ref{sec:illust_ex}.

\subsubsection{Comparison of identifiability and IO-identifiability}
\begin{remark}[Meaning of IO-identifiability]
One can see that the field of IO-identifiable functions is exactly what will be computed by the first two steps of the algorithm outlined in the introduction (see also Algorithm~\ref{alg:main}).
The \textbf{general problem} as stated in Section~\ref{subsec:main_problem} can be restated as: \begin{displayquote}\begin{center}\emph{Determine classes of ODE models for which}
    \text{\emph{identifiable}} $\iff$ \text{\emph{IO-identifiable}}.
\end{center}
  \end{displayquote}
\end{remark}
\cite[Theorem~4.2]{ident-compare} together with~\cite[Example~2.14]{Hong} (see also  
Section~\ref{sec:ex_nonsolvable}
and \cite[Sections~5.2 and~5.3]{allident} with non-constant dynamics and outputs) imply that: 
\begin{equation}\label{eq:identincl}
  \boxed{\text{Identifiable}} \subsetneq \boxed{\text{IO-identifiable}}.
\end{equation}

%%%%%%%%%%%%%%%%%%%%%%%%%%%%%%%%%%%%%%%%

\subsection{Linear compartment models}\label{sec:lincomp}
In this section, we discuss linear compartment models~\cite{LinCompBook}. 
Such a model consists of a set of compartments in which material is transferred  from some compartments to other compartments.  
We also allow for leakage of material from some compartments out of the system, and for input of material into some compartments from outside the system.\hypertarget{GraphG}{}

We use the notation of~\cite[Section~2]{Meshkat14res} but our construction will be slightly more general (allowing scaling for inputs and outputs).  
Let $G$ be a simple directed graph with $n$ vertices $V$ and edges $E$.  Let $\In$, $\Out$, and $\Leak$ be subsets of $V$.  
The coefficients of material transfer are 
\[
\{a_{ji} \;|\; j \leftarrow i \in E\}\quad \text{and}\quad \{a_{0i} \;|\; i \in \Leak\},
\]
and there may be some additional parameters, we will denote all the parameters by $\bm{\mu}$ as before.
For $i=1,\ldots,n$, let $x_i$ be the quantity of material in compartment $i$.  
If $i \in \In$, let $b_i(\bm{\mu})u_i$ be the rate at which the experimenter inputs material into the $i$-th compartment, where $b_i \in \mathbb{C}(\bm{\mu}) \setminus \{0\}$.
If $i \in \Out$, let $y_i = c_i(\bm{\mu})x_i$, where $c_i \in \mathbb{C}(\bm{\mu}) \setminus \{0\}$. 
Without loss of generality, we assume \[\Out = \{1,\ldots,m\}.\]  
Now the system of equations governing the dynamics of $x_1,\ldots,x_n$ is given by 
\begin{equation}\label{eq:linear_compartment}
  \Sigma = 
  \begin{cases}
\mathbf{x}' = A(G)\mathbf{x} + \mathbf{u}, \\
y_i = c_i(\bm{\mu})x_i, \quad \text{for every } i \in \Out,
\end{cases}
\hypertarget{ylin}{}
\end{equation}
where $\mathbf{x} = (x_1,\ldots,x_n)^T$, $\mathbf{u}$ is the $n \times 1$ matrix whose $i$-th entry is $b_i(\bm{\mu})u_i$ if $i\in \In$ and $0$ otherwise, and $A(G)$ is the matrix (generalizing the Laplacian of the graph) defined by\hypertarget{AG}{}
\begin{equation}\label{eq:matrixa}
A(G)_{ij} = 
\begin{cases}
-a_{0i} - \sum\limits_{k:i\rightarrow k \in E}a_{ki}, \quad i = j, i \in \Leak \\
- \sum\limits_{k:i\rightarrow k \in E}a_{ki}, \quad i = j, i \not\in \Leak \\
a_{ij}, \quad j \rightarrow i \in E \\
0, \quad \text{otherwise}.
\end{cases}\hypertarget{ulin}{}
\end{equation}
In the notation of~\eqref{eq:sigma}, we have 
\[\mathbf{x} = \{x_1,\ldots,x_n\},\quad
\mathbf{y} = \{y_1,\ldots,y_m\},\quad
\mathbf{u} = \{u_i \;|\; i \in \In\}.
\]
\begin{definition}
A system $\Sigma$ is called a \emph{linear compartment model} if there exists a simple directed graph $G$ with edges $E$ and vertices $V$, subsets $\In$, $\Out$, and $\Leak$ of $V$,  and functions $b_i, c_j \in \mathbb{C}(\bm{\mu}) \setminus \{0\}$ such that 
$\Sigma$ 
has the form of~\eqref{eq:linear_compartment}.
\end{definition}

It was observed in~\cite[Theorem~2]{Meshkat14res} that, for a linear compartment model, one can obtain relations among inputs, outputs, and parameters as follows.
Let $\partial$ be the differentiation operator.
Let $M_{ji}(G)$ be the submatrix of $\partial I - A(G)$ obtained by deleting the $j$-th row and $i$-th column.
Then~\cite[Theorem~2]{Meshkat14res} yields that system~\eqref{eq:linear_compartment} implies that for every  $i \in \Out$,
\begin{equation}\label{eq:io_lincomp}
    \det(\partial I - A) (y_i) - \tfrac{1}{c_i(\bm{\mu})}\sum\limits_{j \in \In} (-1)^{i + j} \det(M_{ji}) (b_j(\bm{\mu})u_j) = 0.
\end{equation}
\cite[Theorem~3.8]{Meshkat18} gives a refined version of~\eqref{eq:io_lincomp} coinciding with~\eqref{eq:io_lincomp} for the cases we consider in our main results.

\begin{definition}[Reachability]\label{def:reach}
We say \emph{vertex $v$ is reachable from vertex $w$} or \emph{one can reach vertex $v$ from vertex $w$} if there exists a directed path from $w$ to $v$.  For example, in the graph $1 \rightarrow 2$, vertex $2$ is reachable from vertex $1$.  We say a leak (resp. input) is reachable from $w$ if there exists a vertex $v$ in $\Leak$ (resp. $\In$) such that $v$ is reachable from $w$.
\end{definition}

\begin{example}\label{ex:lincomp}
Consider the graph
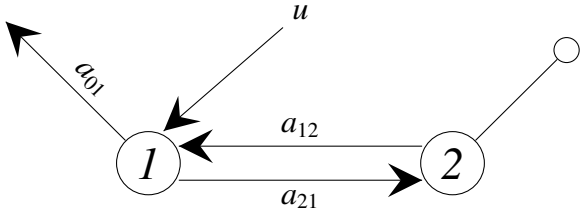
\begin{center}
\begin{tikzpicture}
         \node (1) at (2, 0) [circle, draw=black, font=\LARGE] {1};
         \node (2) at (6, 0) [circle, draw=black, font=\LARGE] {2};
         \node (leak) at (0, 2) [draw=none] {};
         \node (input) at (4, 2) [draw=none, font=\large] {$u$};
         \node (output) at (7.5, 1.5) [circle, draw=black] {};
          \draw[decoration={markings,mark=at position 1 with
    {\arrow[scale=4,>=stealth]{>}}},postaction={decorate}] ([yshift=-2mm,xshift=-4mm]2.north) -- ([yshift=-2mm,xshift=4mm]1.north) node[above, pos=.5, font=\large] {$a_{12}$};
        \draw[decoration={markings,mark=at position 1 with
    {\arrow[scale=4,>=stealth]{>}}},postaction={decorate}] ([yshift=2mm,xshift=4mm]1.south) -- ([yshift=2mm,xshift=-4mm]2.south) node[below, pos=.5, font=\large] {$a_{21}$};
           \draw[decoration={markings,mark=at position 1 with
    {\arrow[scale=4,>=stealth]{>}}},postaction={decorate}] (1) -- (leak) node[above, sloped, pos=.4, font=\large] {$a_{01}$};
            \draw[decoration={markings,mark=at position 1 with
    {\arrow[scale=4,>=stealth]{>}}},postaction={decorate}] (input) -- ([xshift=2mm]1.north);
            \draw[decoration={markings,mark=at position 1 with
    {\arrow[scale=4,>=stealth]{}}},postaction={decorate}] (output) -- (2);
      \end{tikzpicture}
\end{center}
Here $G$ is the graph given by \[V = \{1,2\}\quad \text{and}\quad E = \{1 \rightarrow 2, \; 2 \rightarrow 1\}.\]  The arrow leaving compartment $1$ indicates that $\Leak = \{1\}$, the arrow entering compartment $1$ indicates that $\In = \{ 1 \}$, and the other decoration to compartment $2$ indicates that $\Out = \{2\}$. Note that the input and leak arrows, as well as the output decoration, are not considered part of the graph.
One can see that the corresponding system of differential equations coincides with~\eqref{eq:2comp} and can be written as
\begin{equation*}
\begin{pmatrix}x_1 \\ x_2\end{pmatrix}' = \begin{pmatrix}-(a_{01}+a_{21}) & a_{12} \\ a_{21} & -a_{12} \end{pmatrix}
\begin{pmatrix}x_1 \\ x_2\end{pmatrix} + 
\begin{pmatrix}u \\ 0\end{pmatrix}, \ y = x_2.
\end{equation*}
One can see that this system satisfies the conditions of Theorems~\ref{thm:main1}, \ref{thm:main2}, and~\ref{thm:main3}. 
A direct computation shows that the input-output equation~\eqref{eq:io_example} is a special case of~\eqref{eq:io_lincomp}.
\end{example}
\subsection{Existing algorithms used in practice yet to be justified}\label{sec:applications_to_algorithms}
 In this section, we will present and justify (rephrasing our Main Results~\ref{thm:main1},~\ref{thm:main2}, and~\ref{thm:main3}) the correctness of two versions  (Algorithms~\ref{alg:main} and~\ref{alg:compartment}) of the algorithm outlines in Section~\ref{sec:background} that were not previously fully justified. Algorithm~\ref{alg:main}  is one of the key components of, e.g., DAISY~\cite{DAISY}, and Algorithm~\ref{alg:compartment} summarizes the approach from~\cite[Definition~3.9]{Meshkat18}.  Our justifications are based on the assumptions stated in Corollaries~\ref{cor:algorithm1} and~\ref{cor:algorithm2}. Omitting some of the assumptions could lead to incorrect conclusions, as we show in Section~\ref{sec:ce}.

\begin{algorithm}[H]
\caption{Computing identifiable functions}\label{alg:main}
\begin{description}
\item[Input] System $\Sigma$ as in~\eqref{eq:sigma}
\item[Output] Generators of the field of identifiable functions of $\Sigma$ (see Corollary~\ref{cor:algorithm1})
\end{description}

\begin{enumerate}[label = \textbf{(Step~\arabic*)}, leftmargin=1mm, align=left]
    \item    Compute a full set $\mathcal{C}$ of input-output equations of $\Sigma$.
    \item Return the coefficients of $\mathcal{C}$
    considered as differential polynomials in $\mathbf{y}$ and $\mathbf{u}$.
\end{enumerate}
\end{algorithm}

\begin{corollary}\label{cor:algorithm1}
  Assume that $\Sigma$ satisfies one of the following: 
  \begin{enumerate}[label=(\arabic*),leftmargin=6mm]
      \item\label{cond:one} $\Sigma$ is as in~\eqref{eq:sigma_lin} and has exactly one output;
      \item\label{cond:lincomp} $\Sigma$ is a linear compartment model such that one can reach a leak or an input from every vertex.
  \end{enumerate}
  Then Algorithm~\ref{alg:main} will produce a correct result for $\Sigma$.
\end{corollary}

\begin{proof}
  Algorithm~\ref{alg:main} will compute generators of the field of IO-identifiable functions.
   Main Results~\ref{thm:main1} and~\ref{thm:main2} imply that, for $\Sigma$ that we consider, the field of IO-identifiable functions coincides with the field of identifiable functions.
\end{proof}

\begin{algorithm}[H]
\caption{Computing identifiable functions}\label{alg:compartment}
\begin{description}
\item[Input] System $\Sigma$ as in~\eqref{eq:sigma} corresponding to a linear compartment model with graph $G$
\item[Output] Generators of the field of identifiable functions of $\Sigma$ (see Corollary~\ref{cor:algorithm2})
\end{description}

\begin{enumerate}[label = \textbf{(Step~\arabic*)},  leftmargin=1mm, align=left]
    \item For every $i \in \Out$, compute  
    an input-output equation
    $p_i$ as in~\eqref{eq:io_lincomp} (or a refined version from~\cite[Theorem~3.8]{Meshkat18}).
    \item Return the coefficients of $\{p_i \mid i \in\Out\}$ considered as differential polynomials in $\mathbf{y}$ and $\mathbf{u}$.
\end{enumerate}
\end{algorithm}

\begin{corollary}\label{cor:algorithm2}
  In the notation of Algorithm~\ref{alg:compartment}, if graph $G$ is strongly connected and has at least one input, then Algorithm~\ref{alg:compartment} will produce a correct result.
\end{corollary}

\begin{proof}
  Follows from  Main Result~\ref{thm:main3}.
\end{proof}

%%%%%%%%%%%%%%%%%%%%%%%%%%%%%%%%%%%%%%%%%%%%%%%%%%%%

\section{Main results}\label{sec:main}
In this section, we will state our main results in a condensed form.  
For the detailed statements, see the corresponding theorems in Section~\ref{sec:proofs}. 
In Section~\ref{sec:applications_to_algorithms}, we show how our main results apply to justifying an algorithm computing all identifiable functions of an ODE model. In 
Section~\ref{sec:illust_ex}, 
we present examples (both of applied and of purely mathematical nature) illustrating  the importance and use of the conditions in the statements of our main results.
Note that, while the first result is restricted to MISO systems, the second and third are applicable to MIMO systems as well.

\begin{thm}[see Theorem~\ref{thm:linearIO}]\label{thm:main1} If system $\hyperlink{Sigma}{\Sigma}$ as in~\eqref{eq:sigma_lin} has exactly one output, then IO-identifiable functions coincide with identifiable functions. 
\end{thm}

\begin{thm}[see Theorem~\ref{thm:lincomp}]\label{thm:main2} If the graph of a linear compartment model is such that one can reach a leak or an input from every vertex, then IO-identifiable functions coincide with identifiable functions.
\end{thm}
\begin{problem}\label{prob:thm2}
Will 
Main Result~\ref{thm:main2} remain true if the condition on the graph is removed or relaxed? 
\end{problem}

In other words, 
 Main Results~\ref{thm:main1} and~\ref{thm:main2} provide classes of models for which the approach via input-output equations outlined in the introduction gives the correct result.

\begin{thm}[see Theorem~\ref{thm:compcoeff}]\label{thm:main3} 
  For a linear compartment model
   with at least one input and
   whose graph is strongly connected,
  the field of all identifiable functions is generated by the coefficients of equations~\eqref{eq:io_lincomp}.
\end{thm}

This theorem combined with Lemma~\ref{lem:trans_io} yields:

\begin{corollary}\label{cor:trans}
  For a linear compartment model satisfying the assumptions of Main Result~\ref{thm:main3},
  the field of all identifiable functions is generated by the coefficients of the entries of the transfer function matrix (see Section~\ref{sec:transfer} in Appendix).
\end{corollary}

%%%%%%%%%%%%%%%%%%%%%%%%%%%%%%%%%%%%%%%%%%%%%%%%%%%%%%%%%%%%%%%

\section{Examples}\label{sec:illust_ex}

\subsection{How identifiability methods could make mistakes}\label{sec:ce}
In this section, we will consider several examples to illustrate 
how the methods based on IO-equations, formula~\eqref{eq:lincompq}, and transfer functions may lead to incorrect conclusions about identifiability.  This is to make the reader more aware of the conditions in our main results.

\subsubsection{Failure to detect non-identifiability with multiple outputs using IO-equations}\label{sec:ex_nonsolvable}

We will discuss a simple example of a linear system 
such that the classical method of IO-equations will not be able to decide on the (non-)identifiability of 2 of the 3 parameters.
The example will also show that the condition of having only one output cannot be removed from our Main~Result~\ref{thm:main1}.

We begin with an ODE for radioactive decay
$x'=-ax$,  with $a$ being an unknown decay rate. 
Suppose now that we have an unknown constant inflow $b$, and so
$x' = -ax+ b$.
Consider the following output (e.g., the radiation level):
$y = cx$,
in which the unknown parameter $c$ represents the properties of the medium between the observer and the radioactive species.

Suppose now that there is a known fixed  outflow $w$ (e.g., through a hole of fixed size),
and so the ODE model becomes
\begin{equation}\label{ex:forthm1}
\begin{cases}
x' = -ax +b - w\\
w' = 0\\
y_1 = cx,\ \ 
y_2 = w
\end{cases}
\end{equation}
We then have 
\begin{equation}\label{eq:IO}
\begin{gathered}y_1' = cx' = -cax +cb - cw=-ay_1 + cb - cy_2\\
y_2'=w'=0,
\end{gathered}
\end{equation}
which can be shown to be a full set of input-output equations.
To check the solvability condition, consider system~\eqref{eq:solvability_example}:
\begin{equation}
\label{eq:forthm1solv}
\begin{cases}
y_1'(t_1) = -ay_1(t_1) + cb - cy_2(t_1)\\
y_1'(t_2) = -ay_1(t_2) + cb - cy_2(t_2)\\
y_1'(t_3) = -ay_1(t_3) + cb - cy_2(t_3),
\end{cases}
\end{equation}
which we consider as a linear system in 
$a$, $cb$, and $c$. 
The matrix of the system is
\[A := \begin{pmatrix}
-y_1(t_1) & 1 & -y_2(t_1)\\
-y_1(t_2) & 1 & -y_2(t_2)\\
-y_1(t_3) & 1 & -y_2(t_3)
\end{pmatrix}\]
Since $y_2$ is a constant, the second and third columns of the matrix are proportional.
Therefore, system~\eqref{eq:forthm1solv} has infinitely many solutions for the corresponding coefficients of the input-output equations, $cb$ and $b$. 
Hence, the matrix is rank-deficient and the solvability condition is not satisfied. 
Therefore, proceeding further with trying to check the identifiability of $b$ and $c$ based just on~\eqref{eq:forthm1solv} could (and will for this example, as we will see) cause an incorrect conclusion as the validity of this method is currently guaranteed under the solvability condition.

For example, the software DAISY (which is based on input-output equations) applied to this model concludes that all of $a$, $b$, and $c$ are globally identifiable. However, neither $b$, nor $c$ is even locally identifiable. This can be seen, e.g., by noticing that the following is an
output-preserving transformation
of system~\eqref{ex:forthm1} for all non-zero $k$:
\begin{align*}
x \to kx,\quad 
c \to \frac{c}{k},\quad 
b \to kb + w - kw.
\end{align*}  Therefore, \hyperlink{assumption}{Assumption~(A)} is not satisfied.
For an analogous example without constant states, see~\cite[Remark~2.15]{Hong}. 

\subsubsection{The transfer function method and generic initial conditions (see also~\cite[Example 10.6]{DiStefanoBook})}
Consider the linear compartment model
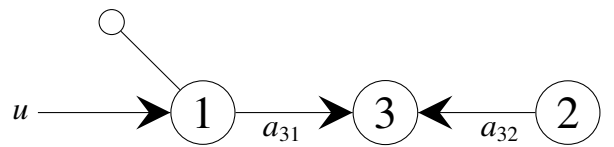
\begin{center}
\begin{tikzpicture}[scale=0.8]
    \node (input) at (0, 0) [draw=none, font=\large] {$u$};
    \node (1) at (3, 0) [circle, draw=black, font=\LARGE] {1};
    \node (3) at (6, 0) [circle, draw=black, font=\LARGE] {3};
    \node (2) at (9, 0) [circle, draw=black, font=\LARGE] {2};
    \node (output) at (1.5, 1.5) [circle, draw=black] {};
    \draw[decoration={markings,mark=at position 1 with
    {\arrow[scale=4,>=stealth]{>}}},postaction={decorate}] (input) -- (1);
    \draw[decoration={markings,mark=at position 1 with
    {\arrow[scale=4,>=stealth]{>}}},postaction={decorate}] (1) -- (3) node[below, pos=.4, font=\large] {$a_{31}$};
    \draw[decoration={markings,mark=at position 1 with
    {\arrow[scale=4,>=stealth]{>}}},postaction={decorate}] (2) -- (3) node[below, pos=.3, font=\large] {$a_{32}$};
    \draw[decoration={markings,mark=at position 1 with
    {\arrow[scale=4,>=stealth]{}}},postaction={decorate}] (output) -- (1);
\end{tikzpicture}
\end{center}
in which an input function $u$ is applied to compartment 1, the quantity in compartment 1 is measured, and material flows from compartment 1 to compartment 3 and from compartment 2 to compartment 3.  The corresponding system is
\begin{equation}\label{eq:eqnforex35}
\begin{gathered}
\begin{pmatrix} x_1 \\ x_2 \\ x_3\end{pmatrix}' = \begin{pmatrix}-a_{31} & 0 & 0 \\ 0 & -a_{32} & 0 \\ a_{31} & a_{32} & 0\end{pmatrix}\begin{pmatrix} x_1 \\ x_2 \\ x_3\end{pmatrix} + \begin{pmatrix}u \\ 0 \\ 0\end{pmatrix} \\ y_1 = x_1.
\end{gathered}
\end{equation}
Since the system satisfies the hypothesis of Theorem~\ref{thm:main1}, we can find the field of identifiable functions, $\mathbb{C}(a_{31})$, by looking at the input-output equation:
\[
  y_1' + a_{31}y_1 - u = 0.
\]
The transfer function (see Section~\ref{sec:transfer} in Appendix) for the system is $\frac{1}{s + a_{31}}$, so the transfer function method gives the same correct result although the initial conditions are not zero but generic  and the assumptions of Theorem~\ref{thm:main3} are not satisfied.
However, using transfer function will lead to erroneous results for this model if we move the output to compartment 3 (that is, replace $y_1 = x_1$ with $y_1 = x_3$).
In this case, the transfer function is $\frac{a_{31}}{s(s + a_{31})}$, indicating that $a_{32}$ is not identifiable.
However, it actually is identifiable.
This can be shown again by using Algorithm~\ref{cor:algorithm1} (as the assumption of Theorem~\ref{thm:main1} is still satisfied), that is, considering the following input-output equation:
\[
y_1''' + (a_{31} + a_{32})y_1'' + a_{31}a_{32}y_1' - a_{31}a_{32}u  - a_{31}u'.
\]
 Thus, the hypotheses of Corollary~\ref{cor:trans} cannot be omitted.
Note that Algorithm~\ref{alg:compartment} will give a correct result for this case even though the assumptions of Theorem~\ref{thm:main3} are not satisfied.

%%%%%%%%%%%%%%%%%%%%%%%%%%%%%%%%%%%%%%%%%%%%%

\subsection{Positive examples for applying our theory}
Below we give examples from the literature satisfying at least one of the sufficient conditions from our main results.
\subsubsection{Kinetics of lead in humans and our 
results
for one output.}
The following system of equations is used in \cite[Section~4A]{LinCompBook} to model the kinetics of lead in the human body:
\begin{equation*}
\begin{cases}
x_1' = k_1x_1+k_2x_2+k_3x_3+k_4 \\
x_2' = k_5x_1+k_6x_2 \\
x_3' = k_7x_1-k_3x_3 \\
y_1 = x_1
\end{cases}
\end{equation*}
A full set of input-output equations is unique in this case and consists of a single differential polynomial:
\begin{multline*}
y_1'''-(k_1+k_3+k_6)y_1'' + (-k_1k_3+k_1k_6-k_2k_5-k_3k_6-k_3k_7)y_1' \\+ (k_1k_3k_6-k_2k_3k_5+k_3k_6k_7)y_1 + k_3k_4k_6.
\end{multline*}
By Corollary~\ref{cor:algorithm1} (condition~\ref{cond:one}), the field of identifiable functions is generated by \begin{gather*}k_1+k_3+k_6,\ \ -k_1k_3+k_1k_6-k_2k_5-k_3k_6-k_3k_7,\\ k_3(k_1k_6-k_2k_5+k_6k_7),\ \ k_3k_4k_6.
\end{gather*}
In other words, these parameter combinations are identifiable, and moreover any other identifiable combination of parameters can be written as a rational combination of these.
\subsubsection{Hepatobiliary kinetics of bromosulfophthalein}
The following linear compartment model is taken from \cite[Section~6.3]{W1982}:
\[
\begin{cases}
x_1' = -k_{31}x_1 + k_{13}x_3 + u\\
x_2' = -k_{42}x_2 + k_{24}x_4\\
x_3' = k_{31}x_1 - (k_{03} + k_{13} + k_{43})x_3\\
x_4' = k_{42}x_2 + k_{43}x_3 - (k_{04} + k_{24})x_4\\
y_1 = x_1,\\ y_2 = x_2.
\end{cases}
\]
A full set of IO-equations is 
too large to display here but their coefficients are
\begin{equation}\label{eq:gens}
k_{13},\ \ k_{31},\ \ k_{04}k_{42},\ \ k_{24}k_{43},\ \  k_{03}+k_{43},\ \ k_{04}+k_{24}+k_{42}.
\end{equation} 
Hence, 
by Corollary~\ref{cor:algorithm1}, 
the field of identifiable functions is generated by~\eqref{eq:gens}. 
This refines the analysis performed in
\cite[Example~3]{MED2009}, where it was shown that~\eqref{eq:gens} generate the field of IO-identifiable functions (as we have seen, for some examples, there are IO-identifiable functions that are not identifiable).

\subsubsection{Cyclic model} 
The following model can be obtained from \cite[Model~M]{GC1990} by adding extra leaks and an output for a better illustration of the computation in connection to our results: 
\[
\begin{cases}
x_1' = a_{13} x_3 - a_{21} x_1 - a_{01}  x_1\\
x_2' = a_{21} x_1 - a_{32} x_2 - a_{02} x_2 + u\\
x_3' = a_{32} x_2 - a_{13} x_3\\
y_1 = x_1,\ \ 
y_2 = x_2
\end{cases}
\]
Using the coefficients of a full set of IO-equations or of the transfer function matrix, we obtain by Corollaries~\ref{cor:algorithm1} and~\ref{cor:trans} that the field of identifiable functions is generated by:
\begin{gather*}
a_{21},\ \ (a_{01} + a_{21})a_{13},\ \ 
a_{01} + a_{13},\ \ 
a_{13}a_{32},\ \ 
a_{02} + a_{32}.
\end{gather*}

%%%%%%%%%%%%%%%%%%%%%%%%%%%%%%%%%%%%%%%%%%%%%%%%%%%%
\section{Proofs}\label{sec:proofs}
\subsection{``$\text{Identifiability} \iff \text{IO-identifiability}$'' for linear systems with one output (proof of Theorem~\ref{thm:main1})}\label{sec:proofs_one_out}

In this section, we prove one of the main results, Theorem~\ref{thm:linearIO}, which shows that, for a linear system with one output, IO-identifiability and identifiability are equivalent. We begin with showing a preliminary result.

\begin{lemma}\label{lem:alglemma}
  Let $K$ be a field. Consider
  \begin{itemize}[leftmargin=7.5mm,leftmargin=4mm]
  \item the differential polynomial ring $K\{y,\mathbf{u}\}$ with derivation $\partial$ satisfying $\partial (K) = 0$,
  \item  $P \in K\{y, \mathbf{u}\}$ of the form  $P = D_P(y) + U_P$,
  where $D_P \in K[\partial]$ is a linear differential operator over $K$ with leading coefficient $1$ and $U_P \in K\{\mathbf{u}\}$.
  \end{itemize}
  Let $W$ be the Wronskian of all the monomials of $P$ except for the one of the highest order with respect to $y$.
  Then $W  \notin [P]$.
\end{lemma}

\begin{proof}
Since the coefficients of $P$ and $W$ are in $K$, the membership $W \in [P]$ would be the same considered over $K$ or its algebraic closure.
Hence, replacing $K$ with its algebraic closure if necessary, we 
assume that $K$ is algebraically closed. 

Consider a lexicographic monomial ordering induced by
an ordering of the variables such that $y^{(i + 1)} > y^{(i)}$ for every $i \geqslant 0$ and $y$ is greater than any derivative of $\mathbf{u}$.
Since for all $r$ $P, P', \ldots, P^{(r)}$ is a Gr\"obner basis for \[[P] \cap K[y,y',\ldots,y^{(r)},\mathbf{u},\mathbf{u}',\ldots,\mathbf{u}^{(r)}],\] it follows from \cite[Lemma 1.5]{Iima2009} that $P, P', \ldots$ form a Gr\"obner basis of $[P]$ with respect to 
this ordering as defined by \cite[Definition 1.4]{Iima2009}.

Since the leading terms of a Gr\"obner basis are linear, $[P]$ is a prime ideal. 
Thus, we can introduce $L := \operatorname{Frac}(K\{y, \mathbf{u}\} / [P])$.
Denote the field of constants of $L$  by $C(L)$
and the images of $y$ and $\mathbf{u}$ in $L$ by $\bar{y}$ and $\bar{\mathbf{u}}$, respectively.
Since none of derivatives of $\mathbf{u}$ appear in the leading terms of the Gr\"obner basis, $\bar{\mathbf{u}}$ and their derivatives are algebraically independent over $K$.

  Assume that the statement of the lemma is not true.
  Due to~\cite[Theorem 3.7, p. 21]{Kap}, this implies that the images in $L$ of the monomials of $P$ except for the 
  one of the highest order in $y$
  are linearly dependent over $C(L)$.
  Therefore, there exists a nonzero polynomial \[Q  = D_Q(y) + U_Q,\] where $D_Q \in C(L)[\partial]$ is monic and $U_Q \in C(L)\{\mathbf{u}\}$, such that  $Q(\overline{y}, \overline{\mathbf{u}}) = 0$ and $\ord D_Q < \ord D_P$.
  Let $D_0$ be the gcd of $D_P$ and $D_Q$ with the leading coefficient 1. 
  Then $\operatorname{ord} D_0 < \operatorname{ord} D_P$.
  
If $F$ is an algebraically closed field, $p \in F[X]$, and $p$ is divisible by a $q \in E[X]$ with the leading coefficient 1, where $E$ is an extension of $F$, then $q \in F[X]$.
  Hence, as $D_0$ divides $D_P$ and $K$ is algebraically closed, $D_0 \in K[\partial]$ and there is $D_1 \in K[\partial]$ such that $D_P = D_1D_0$.
  There also are $A, B \in C(L)[\partial]$ such that \[D_0 = AD_P + BD_Q.\]
  Consider
    $R := A(P) + B(Q) = D_0(y) + U_R$,
  where $U_R = A(U_P) + B(U_Q)$.
  Then $R(\overline{y}, \overline{\mathbf{u}}) = 0$.
  Since $P - D_1(R) \in C(L) \{\mathbf{u}\}$ vanishes on $\overline{\mathbf{u}}$ and $\overline{\mathbf{u}}$ is differentially independent over $C(L)$, it follows that $P = D_1(R)$.
  
  Considering a basis of $C(L)$ over $K$, we can write
  \[
  U_R = U_0 + e_1 U_1 + \ldots + e_N U_N,
  \]
  where $U_0, \ldots, U_N \in K\{\mathbf{u}\}$ and $1, e_1, e_2, \ldots, e_N \in C(L)$ are linearly independent over $K$.
  Since $D_1(U_R) = U_P$ and $D_1 \in K[\partial]$, $U_1, \ldots, U_N \in \ker D_1$, where we consider $D_1$ as a function from $C(L)\{y,\mathbf{u}\}$ to $C(L)\{y,\mathbf{u}\}$.
  There are two cases:
  \begin{itemize}[leftmargin=7.5mm,leftmargin=4mm]
      \item $D_1$ is not divisible by $\partial$.
      Then $\ker D_1 = \{0\}$.  Hence, \[U_1 = \ldots = U_N = 0.\]
      \item $D_1$ is divisible by $\partial$.
      Then $\ker D_1 = C(L)$.
      Thus, $U_1, \ldots, U_N\in K$.
      However, since $U_P = D_1(U_R)$, $U_P$ does not contain a term in $K$. Hence, $U_Q$ does not contain a term in $C(L)$ and, consequently, $U_R$ does not contain a term in $C(L)$.
      Thus, \[U_1 = \ldots = U_N = 0.\]
  \end{itemize}
  In both cases, we have shown that $U_R \in K\{\mathbf{u}\}$.
  Thus, $R \in K\{y, \mathbf{u}\}$ and $R \in [P]$.
  But this is impossible because $P, P', P'', \ldots$ is a Gr\"obner basis of $[P]$ with respect to the monomial ordering introduced in the beginning of the proof, and $\operatorname{ord} D_0 < \operatorname{ord} D_P$, so $R$ is not reducible with respect to this basis.
\end{proof}

\begin{theorem}[Main Result~\ref{thm:main1}]\label{thm:linearIO}
For every $\Sigma$ as in~\eqref{eq:sigma_lin} with $m = 1$ (that is, single output), for all $h \in \CC(\bm{\mu})$, 
\begin{displayquote}\begin{center}
$h$ \text{ is identifiable } $\iff$\ $h$ \text{ is IO-identifiable}. \end{center}
\end{displayquote}
\end{theorem}

\begin{proof}
 \cite[Theorem~4.2]{ident-compare} implies that identifiable functions are always IO-identifiable, so it remains to show the reverse inclusion.
  Consider a full set of input-output equations for $\Sigma$.
  Since $m = 1$, it will consist of a single linear differential polynomial $p \in \mathbb{C}(\bm{\mu}) \{y, \mathbf{u}\}$.
  Then, Lemma~\ref{lem:alglemma} and 
 \cite[Lemma~4.6]{ident-compare} imply that its coefficients are identifiable, so the reverse inclusion holds as well.
\end{proof}

%%%%%%%%%%%%%%%%%%%%%%%%%%%%%%%%%%%%%%%%%%%%%%%%%%%%%%%%%%%%%%%%%%%%

\subsection{Sufficient condition for ``$\text{identifiability} \iff \text{IO-identifiability}$'' for linear compartment models (proof of Theorem~\ref{thm:main2})}
For the notation, see Section~\ref{sec:lincomp}.
\begin{lemma}\label{lem:fconst}
Let $F = \Frac(\CC(\hyperlink{vars}{\bm{\mu}})\{\hyperlink{vars}{\mathbf{x}},\hyperlink{vars}{\mathbf{y}},\hyperlink{vars}{\mathbf{u}}\}/\hyperlink{ISigma}{I_\Sigma})$.  
The field of constants of $F$ lies in the subfield of $F$ generated by $\CC$, $\bm{\mu}$ and~$\mathbf{x}$.
\end{lemma}
\begin{proof}
Observe that $F$ as a field is generated by $\bm{\mu}$, $\mathbf{x}$, and all the derivatives of $\mathbf{u}$, and all these elements are algebraically independent.
Assume that there exists $\ell \geqslant 0$ and $h \in {\CC}(\bm{\mu}, \mathbf{x}, \mathbf{u}, \ldots, \mathbf{u}^{(\ell)})$ such that $h' = 0$ and, without loss of generality, $\frac{\partial}{\partial u_\kappa^{(\ell)}} h \neq 0$.
Then we have
\begin{gather*}
h' = \sum\limits_{i = 0}^\ell\sum\limits_{r = 1}^\kappa u_r^{(i + 1)} \frac{\partial}{\partial u_r^{(i)}} h + \sum\limits_{j = 1}^n x_j' \frac{\partial}{\partial x_j} h  = u_\kappa^{{(\ell + 1)}} \frac{\partial}{\partial u_\kappa^{(\ell)}} h + a,\\
a \in \mathbb{C}(\bm{\mu}, \mathbf{x}, \mathbf{u}, \ldots, \mathbf{u}^{{(\ell)}}, u_1^{(\ell+1)}, \ldots, u_{\kappa-1}^{(\ell+1)}).
\end{gather*}
Now $h' = 0$ {yields} a contradiction since $u_\kappa^{(\ell + 1)}$ is transcendental over $\mathbb{C}(\bm{\mu}, \mathbf{x}, \mathbf{u}, \ldots, \mathbf{u}^{(\ell)}, u_1^{(\ell+1)}, \ldots, u_{\kappa-1}^{(\ell+1)})$ and $\frac{\partial}{\partial u_\kappa^{(\ell)}} h \neq 0$.
\end{proof}

%%%%%%%%%%%%%%%%%%%%%%%%%%%%%%

\begin{lemma}\label{lem:eigen}
Consider a graph $G$ such that, from every vertex, at least one leak can be reached.
Then the eigenvalues of $\hyperlink{AG}{A(G)}$ are distinct and algebraically independent over $\Q$. 
\end{lemma}

\begin{proof}
Let $H$ be a directed spanning forest of $G$ constructed by a breadth-first search (depth-first search would work as well) with the set $\Leak$ as the source such that, from every vertex, there is a path to some element of $\Leak$.
Relabeling vertices if necessary, $A(H)$ is upper triangular with algebraically independent diagonal entries.  
It is well known that a breadth-first search on a graph 
will construct a spanning forest containing all vertices reachable from the source set (cf. \cite[Section 22.2]{Cormen}).  

We illustrate our procedure with an example.  Let $G$ be the graph shown below, with $\Leak = \{1, 6\}$:

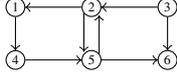
\begin{center}
\begin{tikzpicture}
         \node (1) at (0, .7) [circle, draw=black, inner sep=1pt, font=\tiny] {1};
         \node (2) at (1, .7) [circle, draw=black, inner sep=1pt, font=\tiny] {2};
         \node (3) at (2, .7) [circle, draw=black, inner sep=1pt, font=\tiny] {3};
         \node (4) at (0, 0) [circle, draw=black, inner sep=1pt, font=\tiny] {4};
         \node (5) at (1, 0) [circle, draw=black, inner sep=1pt, font=\tiny] {5};
         \node (6) at (2, 0) [circle, draw=black, inner sep=1pt, font=\tiny] {6};
         \draw[decoration={markings,mark=at position 1 with {\arrow[] {>}}},postaction={decorate}] (2.west) -- (1.east);
         \draw[decoration={markings,mark=at position 1 with {\arrow[] {>}}},postaction={decorate}] (3.west) -- (2.east);
         \draw[decoration={markings,mark=at position 1 with {\arrow[] {>}}},postaction={decorate}] (1.south) -- (4.north);
         \draw[decoration={markings,mark=at position 1 with {\arrow[] {>}}},postaction={decorate}] (3.south) -- (6.north);
         \draw[decoration={markings,mark=at position 1 with {\arrow[] {>}}},postaction={decorate}] (4.east) -- (5.west);
         \draw[decoration={markings,mark=at position 1 with {\arrow[] {>}}},postaction={decorate}] (5.east) -- (6.west);
         \draw[decoration={markings,mark=at position 1 with {\arrow[] {>}}},postaction={decorate}] ([xshift=-1mm,yshift=.5mm]2.south) -- ([xshift=-1mm,yshift=-.2mm]5.north);
         \draw[decoration={markings,mark=at position 1 with {\arrow[] {>}}},postaction={decorate}] ([xshift=1mm,yshift=-.5mm]5.north) -- ([xshift=1mm,yshift=.2mm]2.south);
\end{tikzpicture}
\end{center}
The steps of a breadth-first search with source set $\{1, 6\}$ are the first three upper left, upper right, and lower left graphs shown below.  The fourth lower right graph is a relabeling of the third as described above.
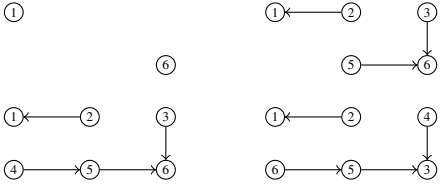
\begin{center}
\begin{tikzpicture}
         \node (1) at (0, .7) [circle, draw=black, inner sep=1pt, font=\tiny] {1};
         \node (6) at (2, 0) [circle, draw=black, inner sep=1pt, font=\tiny] {6};
\end{tikzpicture}
\hspace{10mm}
\begin{tikzpicture}
         \node (1) at (0, .7) [circle, draw=black, inner sep=1pt, font=\tiny] {1};
         \node (2) at (1, .7) [circle, draw=black, inner sep=1pt, font=\tiny] {2};
         \node (3) at (2, .7) [circle, draw=black, inner sep=1pt, font=\tiny] {3};
         \node (5) at (1, 0) [circle, draw=black, inner sep=1pt, font=\tiny] {5};
         \node (6) at (2, 0) [circle, draw=black, inner sep=1pt, font=\tiny] {6};
         \draw[decoration={markings,mark=at position 1 with {\arrow[] {>}}},postaction={decorate}] (2.west) -- (1.east);
         \draw[decoration={markings,mark=at position 1 with {\arrow[] {>}}},postaction={decorate}] (3.south) -- (6.north);
         \draw[decoration={markings,mark=at position 1 with {\arrow[] {>}}},postaction={decorate}] (5.east) -- (6.west);
\end{tikzpicture}
\end{center}
\begin{center}
\begin{tikzpicture}
         \node (1) at (0, .7) [circle, draw=black, inner sep=1pt, font=\tiny] {1};
         \node (2) at (1, .7) [circle, draw=black, inner sep=1pt, font=\tiny] {2};
         \node (3) at (2, .7) [circle, draw=black, inner sep=1pt, font=\tiny] {3};
         \node (4) at (0, 0) [circle, draw=black, inner sep=1pt, font=\tiny] {4};
         \node (5) at (1, 0) [circle, draw=black, inner sep=1pt, font=\tiny] {5};
         \node (6) at (2, 0) [circle, draw=black, inner sep=1pt, font=\tiny] {6};
         \draw[decoration={markings,mark=at position 1 with {\arrow[] {>}}},postaction={decorate}] (2.west) -- (1.east);
         \draw[decoration={markings,mark=at position 1 with {\arrow[] {>}}},postaction={decorate}] (3.south) -- (6.north);
         \draw[decoration={markings,mark=at position 1 with {\arrow[] {>}}},postaction={decorate}] (4.east) -- (5.west);
         \draw[decoration={markings,mark=at position 1 with {\arrow[] {>}}},postaction={decorate}] (5.east) -- (6.west);
\end{tikzpicture}
\hspace{10mm}
\begin{tikzpicture}
         \node (1) at (0, .7) [circle, draw=black, inner sep=1pt, font=\tiny] {1};
         \node (2) at (1, .7) [circle, draw=black, inner sep=1pt, font=\tiny] {2};
         \node (3) at (2, .7) [circle, draw=black, inner sep=1pt, font=\tiny] {4};
         \node (4) at (0, 0) [circle, draw=black, inner sep=1pt, font=\tiny] {6};
         \node (5) at (1, 0) [circle, draw=black, inner sep=1pt, font=\tiny] {5};
         \node (6) at (2, 0) [circle, draw=black, inner sep=1pt, font=\tiny] {3};
         \draw[decoration={markings,mark=at position 1 with {\arrow[] {>}}},postaction={decorate}] (2.west) -- (1.east);
         \draw[decoration={markings,mark=at position 1 with {\arrow[] {>}}},postaction={decorate}] (3.south) -- (6.north);
         \draw[decoration={markings,mark=at position 1 with {\arrow[] {>}}},postaction={decorate}] (4.east) -- (5.west);
         \draw[decoration={markings,mark=at position 1 with {\arrow[] {>}}},postaction={decorate}] (5.east) -- (6.west);
\end{tikzpicture}
\end{center}
Taking $H$ to be the fourth graph, we have 
\[A(H)=
\begin{pmatrix} 
-a_{01} & a_{12} & & & & \\
 & -a_{12} & & & & \\
 & & -a_{03} & a_{34} & a_{35} & \\
 & & & -a_{34} & & \\
 & & & & -a_{35} & a_{56} \\
 & & & & & -a_{56}
\end{pmatrix}.
\]

Since the diagonal entries are algebraically independent over $\mathbb{Q}$ and  algebraic over the field extension of $\Q$ generated by the coefficients of the characteristic polynomial of $A(H)$, it follows that the coefficients of the characteristic polynomial of $A(H)$ are algebraically independent over $\Q$.

For all $i, j$, if the coefficients of the characteristic polynomial of $A(G)|_{a_{i,j}=0}$ are algebraically independent, then the coefficients of the characteristic polynomial of $A(G)$ are algebraically independent.  
Since $A(H)$ can be obtained  from $A(G)$ by setting equal to $0$ those $a_{i,j}$ such that $H$ has no edge from $j$ to $i$, it follows that the coefficients of the characteristic polynomial of $A(G)$ are non-zero and algebraically independent.
Since these $n$ coefficients belong to the field extension of $\mathbb{Q}$ generated by $n$ eigenvalues, the eigenvalues must be algebraically independent as well.
\end{proof}

\begin{theorem}[Main Result~\ref{thm:main2}]\label{thm:lincomp}
Let $\Sigma$ be a linear compartment model with graph $G$ such that, from every vertex of $G$, at least one leak or input is reachable.
Then the fields of identifiable and IO-identifiable functions coincide.
\end{theorem}

\begin{proof}
Let $K := \Frac(\CC(\bm{\mu})\{\mathbf{x},\mathbf{y},\mathbf{u}\}/\hyperlink{ISigma}{I_\Sigma})$.
 We will show that $\Sigma$ does not have a rational first integral, that is $C(K) = \CC(\bm{\mu})$.
 Then the theorem will follow from  \cite[Theorem~4.7]{ident-compare}.
 Consider a model $\Sigma^{\ast}$ with a graph $G^\ast$ obtained from $G$ by replacing every input with a leak (if there was a vertex with an input and a leak, we simply remove the input).
 The theorem will follow from the following two claims.
 
 \emph{Claim: If $\Sigma$ has a rational first integral, then $\Sigma^\ast$ also does.}
 Consider a first integral of $\Sigma$, that is, an element of $C(K) \setminus \CC(\bm{\mu})$.
 Lemma~\ref{lem:fconst} implies that there exists $R\in\CC(\bm{\mu}, \mathbf{x}) \backslash \CC$ such that $c$ is the image of $R$ in $K$.
 Since \[\CC[\bm{\mu}, \mathbf{x}]\{\mathbf{u}\} \cap I_\Sigma = 0\] due to~\cite[Lemma~3.1]{Hong} and the image of $R$ in $K$ is a constant, the Lie derivative of $R$ with respect to $\Sigma$,
 \[
 \mathcal{L}_{\Sigma}(R) := \sum\limits_{i = 1}^n \frac{\partial R}{\partial x_i} f_i,
 \]
where $f_1,\ldots,f_n$ are as in \eqref{eq:sigma}, is zero.
 If there exists $i \in \In$ such that $x_i$ appears in $R$, then $\mathcal{L}_{\Sigma}(R)$ will be of the form
 \[
 \mathcal{L}_{\Sigma}(R) = \frac{\partial R}{\partial x_i} b_i(\bm{\mu})u_i + (\text{something not involving } u_i) \neq 0.
 \]
 Thus, $R$ does not involve any $x_i$ with $i \in \In$.
 Then, due to the construction of $G^\ast$, $\mathcal{L}_{\Sigma^\ast} (R) = \mathcal{L}_{\Sigma} (R) = 0$, so $\Sigma^\ast$ also has a rational first integral.

 \underline{\emph{Claim: $\Sigma^\ast$ does not have rational first integrals.}} 
 Lemma~\ref{lem:eigen} implies that the eigenvalues of $A(G^\ast)$ are algebraically independent.
 Then~\cite[Theorem 10.1.2, p. 118]{nowicki94} implies that $\Sigma^\ast$ does not have rational first integrals.
\end{proof}

%%%%%%%%%%%%%%%%%%%%

\subsection{Using more convenient IO-equations (proof of Theorem~\ref{thm:main3})}
 For the notation, see Section~\ref{sec:lincomp}.
\begin{lemma}\label{lem:poly}
Let $K$ be a field.
For all $a,b,c\in K[x]$  such that $\gcd(a,b) = 1$, there  exists at most one pair $(p, q)$ of elements of $K[x]$ such that
$ap + bq = c$ and $\deg p < \deg b$.
\vspace{0.01in}
\end{lemma}
\begin{proof}
Suppose $(p,q)$ and $(p_1,q_1)$ are distinct pairs satisfying the two properties above.  It follows that \begin{equation}\label{eq:abp}a(p-p_1) + b(q-q_1) = 0.
\end{equation}
Since $(p,q)\ne (p_1,q_1)$,~\eqref{eq:abp} implies that $p\ne p_1$.  Since \[\deg(p-p_1) < \deg b,\]\eqref{eq:abp} implies   $\gcd(a,b) \neq 1$,  contradicting our hypothesis.
\end{proof}

\begin{corollary}\label{cor:poly}
  Let $K$ be a field containing $\CC$ and $a,b,c\in K[x]$  
  with $\gcd(a,b) = 1$.  
  If there is a pair of polynomials $(p,q)$  
   with
  \[ap+bq = c\quad \text{and}\quad \deg p < \deg b,\] then the coefficients of $p$ and $q$ belong to the field extension of $\CC$ generated by the coefficients of $a$, $b$, and $c$.
\end{corollary}
\begin{proof}
Suppose some coefficient of $p$ or $q$ does not belong to the field generated by the coefficients of $a$, $b$, and $c$.  By \cite[Theorem 9.29, p. 117]{milneFT}, there is a field automorphism $\sigma$ of $\overline{K}$ that fixes the field extension of $\CC$ generated by the coefficients of $a$, $b$, and $c$ and moves this coefficient. 

We extend $\sigma$ to $\overline{K}[x]$ by $\sigma(x) = x$.
Applying $\sigma$ to both sides of $ap+bq=c$ gives us 
$a\sigma(p)+b\sigma(q) = c$. 
Using $\overline{K}$ for $K$ in Lemma~\ref{lem:poly}, we arrive at a contradiction.
\end{proof}

\begin{theorem}[Main Result~\ref{thm:main3}]\label{thm:compcoeff}
Let $\Sigma$ be a linear compartment model with a graph $G$.
Let $A = \hyperlink{AG}{A(G)}$ and $M_{ji}$ be the submatrix of $\partial I - A$ obtained by deleting the $j$-th row and the $i$-th column of $\partial I - A$.
Recall that (see~\eqref{eq:io_lincomp}), for every solution of $\Sigma$, we have  for every  $i \in \Out$,
\begin{equation*}
    \det(\partial I - A) (y_i) = \frac{1}{c_i(\bm{\mu})}\sum\limits_{j \in \In} (-1)^{i + j} \det(M_{ji}) (b_j(\bm{\mu})u_j).
\end{equation*}
If 
\hyperlink{GraphG}{$G$} is strongly connected and has at least one input, then the coefficients of these differential polynomials with respect to $y$'s and $u$'s  generate the \hyperlink{idfield}{field of identifiable functions} of $\Sigma$. 
\end{theorem}

\begin{proof}
Without loss of generality, assume $\Out = \{1,\ldots,m\}$.  
We set, for $i = 1, \ldots, m$,
\begin{equation}\label{eq:meshkat}
h_i := \det(\partial I - A)(y_i) - \tfrac{1}{c_i(\bm{\mu})}\sum_{j \in \In}(-1)^{i+j}\det(M_{ji})b_j(\bm{\mu})u_j. \hypertarget{MatA}{}
\end{equation}
Let also $D = \det(\partial I - A)$ and, for $i = 1, \ldots, m$, let $Q_i$ be the $1 \times n$ matrix of operators defined by 
\begin{equation}\label{eq:lincompq}
\begin{cases}(Q_i)_j = \frac{(-1)^{i+j}}{c_i(\bm{\mu})}\det(M_{ji}) b_j(\bm{\mu}),&j \in \In,\\
(Q_i)_j = 0, &j \notin \In.
\end{cases}
\end{equation}
Observe that, for $i=1,\ldots,m$
$h_i = D(y_i) - Q_i \cdot \mathbf{u}$,
where $\mathbf{u}$ is the $n \times 1$ matrix defined by $\mathbf{u}_j = u_j$ if $j \in \In$ and $\mathbf{u}_j = 0$ otherwise.

First we show that the coefficients of $h_1, \ldots, h_m$ are IO-identifiable. 
Fix $i$.  
Consider an ordering of the outputs such that $y_i$ is the smallest one.
Let $p_1, \ldots, p_m$ be a full set of input-output equations with respect to this ordering (see Definition~\ref{def:ioeq}) which exists due to Proposition~\ref{prop:linear_equiv}.
Then $p_1$ is of the form \[
E(y_i) + B \cdot \mathbf{u},\] where $E$ is a linear differential operator and $B$ is a $1 \times n$ matrix of linear differential operators, both with  coefficients in $\CC(\bm{\mu})$.
Since $h_i \in I_{\Sigma}$ and $h_i$ involves only $y_i$ and $\mathbf{u}$, the second part of Proposition~\ref{prop:linear_equiv} implies that $h_i \in [p_1]$, so there exists a differential operator $D_0 \in \CC(\bm{\mu})[\partial]$ such that $h_i = D_0p_1$.
 Since $G$ is strongly connected and has an input, by \cite[Proposition~3.19]{Meshkat18}, 
 \[
 \gcd(D \cup \{(Q_i)_j \mid (Q_i)_j \neq 0\}) = 1.
 \]  
 Thus $D_0$ has order zero, so $h_i$ and $p_1$ are proportional. 
 Therefore, the coefficients of \eqref{eq:meshkat} are IO-identifiable.

Next, we show that the field generated by the coefficients of $h_1,\ldots,h_m$ contains the field of IO-identifiable functions.  
Fix an ordering on the outputs $ y_m > \ldots > y_1$. 
We will show that the full set $p_1, \ldots, p_m$ of input-output equations with respect to this ordering satisfies:
\begin{equation}
    \ord_{y_1} p_1 = n, \quad \ord_{y_i} p_i = 0 \text{ for every } 2 \leqslant i \leqslant m.
\end{equation}
The fact that $\ord_{y_1} p_1 = n$ is implied by the previous paragraph.
From \eqref{eq:sigma_lin}, we see that the transcendence degree of \[\CC(\bm{\mu})\{\mathbf{x},\mathbf{y},\mathbf{u}\}/I_\Sigma\] over $\CC(\bm{\mu})\{\mathbf{u}\}$ is equal to $n$, so the transcendence degree of \[\CC(\bm{\mu})\{\mathbf{y},\mathbf{u}\}/(I_\Sigma \cap \CC(\bm{\mu})\{\mathbf{y},\mathbf{u}\})\] over $\CC(\bm{\mu})\{\mathbf{u}\}$ is less than or equal to $n$.  
From the form of $p_1$, we have that $y_1,y_1',\ldots,y_{1}^{(n-1)}$ are algebraically independent over $\CC(\bm{\mu})\{\mathbf{u}\}$, so for $i=2,\ldots,m$, the elements 
$y_i, y_1, y_1', \ldots,y_1^{(n-1)}$ 
must be algebraically dependent over $\CC(\bm{\mu})\{\mathbf{u}\}$.  Hence, the equation for $y_i$ has order $0$ in $y_i$.
Thus,  
\[
p_1 = D(y_1) - Q_1 \cdot \mathbf{u}
\] and, for every, $2 \leqslant i \leqslant m$, we can write 
\[
p_i = y_i + D_i(y_1) + P_i \cdot \mathbf{u},
\]
where $P_i$ is a $1 \times n$ matrix of linear differential operators and the order of operator $D_i$ is at most $n - 1$.

We show that the coefficients of $p_1, \ldots, p_m$ can be written in terms of the coefficients of $h_1,\ldots,h_m$.  
Since $h_1$ equals $D(y_1) - Q_1 \cdot \mathbf{u}$, this is true for the coefficients of $D$ and $Q_1$.  
It remains to show this for the coefficients of $D_2,\ldots,D_m$ and $P_2,\ldots,P_m$.  
Note that for all $i$ and for $j \not\in \In$ we have $(P_i)_j = 0$, so we need only address the coefficients of $(P_i)_j$ for $j \in \In$.

Fix $i > 1$ and let $g = y_i + D_i(y_1) + P_i(\mathbf{u})$.  
We have that \[D(g) - D_i(h_1) = D(y_i) + (DP_i + D_iQ_1)(\mathbf{u}) \in I_\Sigma.\]  It follows that $D(y_i) + (DP_i + D_iQ_1)(\mathbf{u}) = h_i$, so, for all $j$,
\[D(P_i)_j + D_i(Q_1)_j = -(Q_i)_j.\]
By the hypothesis of the theorem, $\In \neq \varnothing$.  
Fix $j \in \In$.  
We apply~\cite[Proposition~3.19]{Meshkat18} to the model obtained from $\Sigma$ by deleting all the inputs except for $j$ and obtain, using  $D \neq 1$, that
$
  \gcd(D, (Q_1)_j) = 1$ for every  $j \in \In$.
By Corollary~\ref{cor:poly}, 
the coefficients of $(P_i)_j$ and $D_i$ belong to the field extension of $\CC$ generated by the coefficients of $D$, $(Q_1)_j$, and $(Q_i)_j$.
We  showed that the field extension of $\CC$ generated by the coefficients of $h_1,\ldots,h_m$ is  the field of IO-identifiable functions.  By Theorem~\ref{thm:lincomp}, this is the field of identifiable functions.
\end{proof}

%%%%%%%%%%%%%%%%%%%%%%%%%%%%%%%%%%%%%%%%%%%%%%%%%%%%%%%%%%%%

\appendix\section{General facts about identifiability and IO-identifiability}\hypertarget{AppendixA}{}\label{app:general}

\subsection{General definition of identifiability}\label{sec:ident_general}

In this section, we will generalize the notions from Section~\ref{sec:prelim_ident} to ODE systems with rational right-hand side.
\hyperlink{vars}{}
Fix positive integers $\lambda$, $n$, $m$, and $\kappa$ for the remainder of the appendix.
Let $\bm{\mu} = (\mu_1,\ldots,\mu_\lambda)$, $\mathbf{x} = (x_1, \ldots, x_n)$, $\mathbf{y} = (y_1, \ldots, y_m)$, and $\mathbf{u} = (u_1,\ldots,u_\kappa)$.
Consider a system of ODEs \hypertarget{Sigma}{}
\begin{equation}\label{eq:sigma}
\Sigma = 
\begin{cases}
\mathbf{x}' = \cfrac{\mathbf{f}(\mathbf{x}, \bm{\mu}, \mathbf{u})}{Q(\mathbf{x}, \bm{\mu}, \mathbf{u})}, \\
\mathbf{y} = \cfrac{\mathbf{g}(\mathbf{x}, \bm{\mu}, \mathbf{u})}{Q(\mathbf{x}, \bm{\mu}, \mathbf{u})},\\
\mathbf{x}(0) = \mathbf{x}^\ast,
\end{cases}
\end{equation}
where $\mathbf{f} = (f_1, \ldots, f_n)$ and $\mathbf{g} = (g_1, \ldots, g_m)$ are tuples of elements of $\CC[\bm{\mu},\mathbf{x},\mathbf{u}]$ and $Q\in\CC[\bm{\mu},\mathbf{x},\mathbf{u}] \backslash \{0\}$.

\begin{notation}[Ideal $I_\Sigma$]
\begin{enumerate}[label = (\alph*),leftmargin=7.5mm]
\item[]
    \item \hypertarget{colon}{}\hypertarget{infinity}{}
  For an ideal $I$ and element $a$ in a ring $R$, we denote \[I \colon a^\infty = \{r \in R \mid \exists \ell\colon a^\ell r \in I\}.\]
  This set is also an ideal in $R$.
  \hypertarget{ISigma}{}
    \item Given $\Sigma$ as in~\eqref{eq:sigma}, we define the differential ideal of $\Sigma$: \[I_\Sigma=[Q\mathbf{x}'-\mathbf{f},Q\mathbf{y}-\mathbf{g}]:Q^\infty \subset \CC(\bm{\mu})\{\mathbf{x},\mathbf{y},\mathbf{u}\}.\]
    For the case of a linear system as in~\eqref{eq:sigma_lin}, this ideal coincides with the one from Notation~\ref{not:lin_diffalg}.
\end{enumerate}
\end{notation}

\begin{notation}[Auxiliary analytic notation]
\begin{enumerate}[label = (\alph*),leftmargin=5.5mm]
\item[]
    \item 
    For every given $h \in \CC(\mathbf{x}^\ast, \bm{\mu})$, let 
    \begin{align*}\Omega=\{(\hat{\mathbf{x}}^*,\hat{\bm{\mu}},\hat{\mathbf{u}})& \in \CC^n \times \CC^\lambda \times (\CC^\infty(0))^\kappa \mid Q(\hat{\mathbf{x}}^*,\hat{\bm{\mu}},\hat{\mathbf{u}}(0)) \ne 0\}\\
    \Omega_h = \Omega \cap (&\{(\hat{\mathbf{x}}^\ast, \hat{\bm{\mu}}) \in \CC^{n + \lambda} \mid h(\hat{\mathbf{x}}^\ast, \hat{\bm{\mu}}) \text{ well-defined}\}\\ &\times (\CC^\infty(0))^\kappa).
    \end{align*}
    \item For $(\hat{\mathbf{x}}^\ast, \hat{\bm{\mu}}, \hat{\mathbf{u}}) \in \Omega$, let $X(\hat{\mathbf{x}}^\ast, \hat{\bm{\mu}}, \hat{\mathbf{u}})$ and $Y(\hat{\mathbf{x}}^\ast, \hat{\bm{\mu}}, \hat{\mathbf{u}})$ denote the unique solution over $\CC^\infty(0)$ of the instance of $\Sigma$ with $\mathbf{x}^\ast = \hat{\mathbf{x}}^\ast$, $\bm{\mu} = \hat{\bm{\mu}}$, and $\mathbf{u} = \hat{\mathbf{u}}$ (see~\cite[Theorem~2.2.2]{Hille}).
    
\end{enumerate}
\end{notation}

\begin{definition}[Identifiability, see {\cite[Definition~2.5]{Hong}}]\label{def:seid}
We say that $h(\mathbf{x}^\ast,   \bm{\mu}) \in \CC(\mathbf{x}^\ast,  \bm{\mu})$ is \emph{identifiable} if 
\begin{gather*}
  \exists \Theta \in \tau(\CC^n \times \CC^\lambda) \; \exists U \in \tau((\CC^\infty(0))^\kappa)\;\\
  \forall (\hat{\mathbf{x}}^\ast, \hat{\bm{\mu}}, \hat{\mathbf{u}}) \in (\Theta \times U) \cap \Omega_h \quad |S_h(\hat{\mathbf{x}}^\ast, \hat{\bm{\mu}}, \hat{\mathbf{u}})| = 1,\\
 \text{where}\ \ 
  S_h(\hat{\mathbf{x}}^\ast, \hat{\bm{\mu}}, \hat{\mathbf{u}}) := \{h(\tilde{\mathbf{x}}^\ast,  \tilde{\bm{\mu}}) \mid  (\tilde{\mathbf{x}}^\ast, \tilde{\bm{\mu}}, \hat{\mathbf{u}}) \in \Omega_h\\ \quad\quad Y(\hat{\mathbf{x}}^\ast, \hat{\bm{\mu}}, \hat{\mathbf{u}}) = Y(\tilde{\mathbf{x}}^\ast, \tilde{\bm{\mu}}, \hat{\mathbf{u}}) \}.
  \end{gather*}
In this paper, we are interested in comparing identifiability and IO-identifiability (Definition~\ref{def:ioid}), and the latter is defined 
for functions in $\bm{\mu}$, not in $\bm{\mu}$ and $\mathbf{x}^\ast$.
Thus, just for the purpose of comparison, we will restrict ourselves to the field \[\{h \in \mathbb{C}(\bm{\mu})\mid h \text{ is identifiable}\},\] which we will call \emph{the field of identifiable functions}.
\end{definition}

\begin{definition}[IO-identifiability]\label{def:ioid}\hypertarget{fieldio}{}
The smallest field $k$ such that $\CC \subset k \subset \CC(\bm{\mu})$ and  $\hyperlink{ISigma}{I_\Sigma} \cap \CC(\bm{\mu})\{\mathbf{y},\mathbf{u}\}$ is generated (as an ideal or as a differential ideal) by $I_\Sigma \cap k\{\mathbf{y},\mathbf{u}\}$ is called \emph{the field of IO-identifiable functions}.

We call $h \in \CC(\bm{\mu})$  \emph{IO-identifiable} if $h \in k$.
\end{definition}

%%%%%%%%%%%%%%%%%%%%%%%%%%%%%%%%%%%%%%%%%%%%%%%%%%%%%

\subsection{Specialization to the linear case}\label{sec:spec_linear}

\begin{proposition}\label{prop:linear_equiv}
  For every system $\Sigma$ of the form~\eqref{eq:sigma_lin}:
  \begin{enumerate}[label=(\arabic*)]
      \item\label{prA6:part1} for every ordering of output variables, there exists a unique full set of input-output equations with respect to this ordering;
      \item\label{prA6:part2} if $p_1, \ldots, p_m$ is the full set of input-output equations with respect to $y_1 < \ldots < y_m$, then the derivatives of $p_1, \ldots, p_m$ form a Gr\"obner basis of $I_{\Sigma} \cap \mathbb{C}(\bm{\mu}) \{\mathbf{y}, \mathbf{u}\}$ with respect to any lexicographic monomial ordering such that 
  \begin{itemize}
        \item any derivative of any of $y$'s is greater any derivative of any of $u$'s;
        \item $y_{i_1}^{(j_1)} > y_{i_2}^{(j_2)}$ iff $i_1 > i_2$ or $i_1 = i_2$ and $j_1 > j_2$.
    \end{itemize}
    An analogous statement holds for  any 
    ordering of outputs.
    \item\label{prA6:part3} Definitions~\ref{def:ioid} and~\ref{def:ioid_lin} define the same field.
      In particular, the field defined in Definition~\ref{def:ioid_lin} does not depend on the choice of a full set of input-output equations. 
  \end{enumerate}
\end{proposition}

\begin{proof}
    We fix an ordering $y_1 < \ldots < y_m$ of outputs. 
    Assume that there are full sets of input-output equations $p_1, \ldots, p_m$ and $q_1, \ldots, q_m$ with respect to this ordering. 
    Let $\ell$ be the smallest integer such that $p_\ell \neq q_\ell$.
    By the definition, $\ord_{y_\ell} p_\ell = \ord_{y_\ell} q_\ell$.
    Then $\ord_{y_i} (p_\ell - q_\ell) < \ord_{y_i} p_i$ for every $i \leqslant \ell$; this contradicts  the definition of a full set of input-output equations.
    To finish the proof  of
    part~\ref{prA6:part1} of the proposition,  we will
    show the existence of a full set of input-output equations.
 
    Let $J := I_{\Sigma} \cap \mathbb{C}(\bm{\mu})\{\mathbf{y}, \mathbf{u}\}$.
    Consider the set of differential polynomials \[S := \{\mathbf{x}' - \mathbf{f}, \mathbf{x}'' - \mathbf{f}', \ldots, \mathbf{y} - \mathbf{g}, \mathbf{y}' - \mathbf{g}', \ldots\}.\]
    By the definition of $I_{\Sigma}$, 
    $S$
    generates $I_{\Sigma}$.
    Since these generators
    are linear, $I_{\Sigma}$ has a linear Gr\"obner basis (see \cite[Definition 1.4]{Iima2009}) with respect to any monomial ordering.
    Since $J$ is an elimination ideal of $I_{\Sigma}$, it also has a linear Gr\"obner basis with respect to any monomial ordering.
    Consider any lexicographic monomial ordering on $\mathbb{C}(\bm{\mu})\{\mathbf{x}, \mathbf{y}, \mathbf{u}\}$ such that
    \begin{itemize}
        \item any derivative of any of $y_1,\ldots,y_m$ is greater than any derivative of any of $x_1,\ldots,x_n$;
        \item any derivative of any of $x_1,\ldots,x_n$ is greater than any derivative of any of  $u_1,\ldots,u_\kappa$;
        \item for $a = x, y$, $a_{i_1}^{(j_1)} > a_{i_2}^{(j_2)}$ iff $i_1 > i_2$ or $i_1 = i_2$ and $j_1 > j_2$.
    \end{itemize}
    Observe that $S$ is a Gr\"obner basis of $I_{\Sigma}$ with respect to any such monomial ordering.
    Therefore, $\mathbf{u}$ and their derivatives are algebraically independent modulo $I_{\Sigma}$, and the transcendence degree of $\mathbb{C}(\bm{\mu})\{\mathbf{x}, \mathbf{y}, \mathbf{u}\}$ over $\mathbb{C}(\bm{\mu}) \{\mathbf{u}\}$ modulo $I_{\Sigma}$ is finite.
    
    Consider the restriction of the ordering described above to $\mathbb{C}(\bm{\mu})\{\mathbf{y}, \mathbf{u}\}$.
    Consider the reduced Gr\"obner basis $B$ of $J$ with respect to this ordering.
    As we have shown, it is linear.
    Since the transcendence degree of $\mathbb{C}(\bm{\mu})\{\mathbf{y}, \mathbf{u}\}$ over $\mathbb{C}(\bm{\mu})\{\mathbf{u}\}$ modulo $J$ is finite, for every $1 \leqslant i \leqslant m$, there is a derivative of $y_i$ among the leading terms of $B$.
    Moreover, by differentiating the corresponding element of $B$, we see that all higher derivatives of $y_i$ will appear as leading terms of $B$.
    
    For each $1 \leqslant i \leqslant m$, we set $p_i$ to be the element in $B$ with the leading term being $y_i^{(j)}$ such that $j$ is the smallest possible.
    Then the fact that $p_1, \ldots, p_m$ are a part of the reduced Gr\"obner basis implies that they form a full set of input-output equations with respect to the ordering $y_1 < y_2 <  \ldots < y_m$.
    This finishes the proof of 
    part~\ref{prA6:part1} of the proposition.
    
    To prove
    part~\ref{prA6:part2} of the proposition, observe that the derivatives of $p_1, \ldots, p_m$ form a Gr\"obner basis of $[p_1, \ldots, p_m]$ with respect to the described ordering.
    Thus, it remains to show that $[p_1, \ldots, p_m] = J$.
    Assume that there is $q \in J \setminus [p_1, \ldots, p_m]$.
    By reducing it with respect to appropriate derivatives of $p_1, \ldots, p_m$, we can assume that $\ord_{y_i} q < \ord_{y_i} p_i$ for every $1 \leqslant i \leqslant m$.
    But this would imply that $p_1, \ldots, p_m$ is not a full set of input-output equations,  proving
    part~\ref{prA6:part2} of the proposition.
    
    To prove 
    part~\ref{prA6:part3} of the proposition, note that, since a full set of input-output equations is a part of a reduced Gr\"obner basis of $J$, its coefficients 
   are in
    the field of definition of $J$.
    On the other hand, since the set of all derivatives of $p_1, \ldots, p_m$ forms a Gr\"obner basis of $J$ and the coefficients of these derivatives are the same 
    as the coefficients of $p_1, \ldots, p_m$, the coefficients of $p_1, \ldots, p_m$ generate the field of definition of $J$.
\end{proof}

%%%%%%%%%%%%%%%%%%%%%%%%%%%%%%%%%%%%%%%%%%

\subsection{Input-output equations based on the Cramer's rule and the transfer function matrix}\label{sec:transfer}

Recall~\cite[page~444]{DiStefanoBook} that the \emph{transfer function matrix} of  
 a linear system
\[
\begin{cases}
  \mathbf{x}' = A(\bm{\mu})\mathbf{x} + B(\bm{\mu}) \mathbf{u},\\
  \mathbf{y} = C(\bm{\mu}) \mathbf{x}.
\end{cases}
\]
is defined by
\begin{equation}\label{eq:trans_func_def}
H(\bm{\mu}, s) := C(\bm{\mu}) (sI - A(\bm{\mu}))^{-1}B(\bm{\mu}),
\end{equation}
where $s$ is a new algebraic variable and $I$ is the identity matrix.
This matrix relates the Laplace transforms of $\mathbf{y}$ and $\mathbf{u}$ under the assumption that the initial conditions are zero~\cite[page~75]{DiStefanoBook}.
The formulas~\eqref{eq:io_lincomp} and~\eqref{eq:trans_func_def} look similar and in fact are related.
We give a connection we are interested in as Lemma~\ref{lem:trans_io} below, and refer for further connection to an upcoming paper~\cite{NikkiMarisa}.

For a rational function $f \in \mathbb{C}(\bm{\mu})(s)$ in $s$, by \emph{the coefficients of $f$}, we will understand the union of the coefficients of the numerator and denominator \emph{in the reduced form} if the denominator is taken to be monic.

\begin{lemma}\label{lem:trans_io}
  Consider a linear compartment model with at least one input and whose graph is strongly connected. 
  Then the following sets generate the same subfield in $\mathbb{C}(\bm{\mu})$
  \begin{itemize}
      \item coefficients of the input-output equations~\eqref{eq:io_lincomp};
      \item coefficients of the entries of the transfer function matrix.
  \end{itemize}
\end{lemma}

\begin{proof}
    Since each of the equations~\eqref{eq:io_lincomp} involves only one output and each row in the matrix~\eqref{eq:trans_func_def} corresponds to an output, proving the lemma for the single-output case will yield the general case by taking the union of the respective generators.
    
    We write~\eqref{eq:io_lincomp} as $p(\partial)y = q_1(\partial)u_1 + \ldots + q_r(\partial)u_r$ for nonzero $p(s), q_1(s), \ldots, q_r(s) \in \mathbb{C}(\bm{\mu})[s]$ such  that $p(s)$ is monic.
    Let $F_1$ be the field generated by the coefficients of $p, q_1, \ldots, q_r$.
    Since the graph is strongly connected and has an input, \cite[Proposition~3.19]{Meshkat18} implies that $\gcd(p, q_1, \ldots, q_r) = 1$.
    A direct computation shows that $H(\bm{\mu})$ defined by~\eqref{eq:trans_func_def} is equal to 
    \[
    H(\bm{\mu}) := (h_1(s),h_2(s),\ldots,h_r(s)) = \left(\frac{q_1(s)}{p(s)}, \frac{q_2(s)}{p(s)}, \ldots, \frac{q_r(s)}{p(s)}   \right).
    \]
    Let $F_2$ be the field generated by the coefficients of $h_1, \ldots, h_r$.
    For all integers $n_1, \ldots, n_r$, the coefficients of $n_1h_1 + \ldots + n_rh_r$ belong to $F_2$.
    Since $q_1, \ldots, q_r, p$ are coprime, there exist integers $n_1,\ldots,n_r$ such that $n_1q_1 + \ldots + n_rq_r$ is coprime with $p$, so $p$ is the denominator of $n_1h_1 + \ldots + n_rh_r$.
    Hence, the coefficients of $p$ belong to $F_2$.
    Let $g_i = \gcd(p,q_i)$ and $p_i = p/g_i$.  
    By definition, the coefficients of $p_i$ are in $F_2$, so the coefficients of $g_i$ are in $F_2$.  
    By definition, the coefficients of $q_i/g_i$ are in $F_2$, so the coefficients of $q_i$ are in $F_2$.
    Thus, $F_1 \subset F_2$.
    
    To prove $F_2 \subset F_1$, note that the coefficients of the remainder and quotient of two polynomials belong to the field generated by the coefficients of the polynomials.
    Since the numerator and denominator of $h_i$ are equal to $q_i / \gcd(p, q_i)$ and $p / \gcd(p, q_i)$, respectively, we have $F_2 \subset F_1$, so $F_1 = F_2$.
\end{proof}

%%%%%%%%%%%%%%%%%%%%%%%%%%%%%%%%%%%%%%%%%%

\section*{Acknowledgments}We are grateful to the CCiS at CUNY Queens College for the computational resources and to Julio Banga,  Joseph DiStefano, Marisa Eisenberg, Nikki Meshkat, Maria Pia Saccomani, Anne Shiu, Seth Sullivant, Alejandro Villaverde, and to referees for useful discussions and suggestions.

%%%%%%%%%%%%%%%%%%%%%%%%%%%%%%%%%%%%%%%%%%%%%%%%%%%%%%%%%%

\bibliographystyle{IEEEtran}
\bibliography{bibdata}
\end{document}